\documentclass[11pt]{amsart}
\usepackage{amsmath}
\usepackage{amssymb}

\usepackage[all]{xy}%
\theoremstyle{plain}
    \newtheorem{theorem}{Theorem}[section]
    \newtheorem{proposition}[theorem]{Proposition}
    \newtheorem{lemma}[theorem]{Lemma}
    
\theoremstyle{definition}
    \newtheorem{conjecture}{Conjecture}
    \newtheorem{definition}[theorem]{Definition}
\theoremstyle{remark}
    \newtheorem*{notation}{Notation}
    
    \newtheorem{remark}[theorem]{Remark}
\def\Alphabet{A,B,C,D,E,F,G,H,I,J,K,L,M,N,O,P,Q,R,S,T,U,V,W,X,Y,Z}
\def\alphabet{a,b,c,d,e,f,g,h,i,j,k,l,m,n,o,p,q,r,s,t,u,v,w,x,y,z}
\def\endpiece{xxx}
\def\makeAlphabet[#1]{\expandafter\makeA#1,xxx,}
\def\makealphabet[#1]{\expandafter\makea#1,xxx,}
\def\makeA#1,{\def\temp{#1}\ifx\temp\endpiece\else%
\mkbb{#1}\mkfrak{#1}\mkbf{#1}\mkcal{#1}\mkscr{#1}\expandafter\makeA\fi}%
\def\makea#1,{\def\temp{#1}\ifx\temp\endpiece\else\mkfrak{#1}\mkbf{#1}\expandafter\makea\fi}%
\def\mkbb#1{\expandafter\def\csname bb#1\endcsname{\mathbb{#1}}}
\def\mkfrak#1{\expandafter\def\csname fr#1\endcsname{\mathfrak{#1}}}
\def\mkbf#1{\expandafter\def\csname b#1\endcsname{\mathbf{#1}}}
\def\mkcal#1{\expandafter\def\csname c#1\endcsname{\mathcal{#1}}}
\def\mkscr#1{\expandafter\def\csname s#1\endcsname{\mathscr{#1}}}
\def\makeop[#1]{\xmakeop#1,xxx,}
\def\mkop#1{\expandafter\def\csname #1\endcsname{{\mathrm{#1}}\,}} %
\def\xmakeop#1,{\def\temp{#1}\ifx\temp\endpiece\else\mkop{#1}\expandafter\xmakeop\fi}%
\def\makesymb[#1]{\xmakesymb#1,xxx,}
\def\mksymb#1{\expandafter\def\csname #1\endcsname{{\mathrm{#1}}}} %
\def\xmakesymb#1,{\def\temp{#1}\ifx\temp\endpiece\else\mksymb{#1}\expandafter\xmakeop\fi}%
\makeAlphabet[\Alphabet]
\makealphabet[\alphabet]
\makeop[Spec, Spm, Spf]
\makesymb[Hom]
\makeop[Pic]
\makeop[Gal]
\makeop[coker, ker]
\makeop[mod]
\makeop[Idem]

\def\Fr{{\rm Fr}}
\def\N{{\rm N}}
\def\Ker{{Ker}}
\def\tors{{\rm tors}}
\def\dlim{\underrightarrow{\rm lim}}

\def\Fr{\mathrm{Fr}}
\def\Ker{\mathrm{Ker}}
\def\Im{\mathrm{Im}}
\def\Id{\mathrm{Id}}
\def\Coker{\mathrm{Coker}}
\def\Gal{\mathrm{Gal}}

\def\dlim{\underrightarrow{\lim}}

\makeop[codim]
\makeop[depth]
\makeop[Pic]
\makeop[mod]
\makeop[Idem]
\makeop[ord]
\makeop[rec]
\makeop[log]

\def\varep{\varepsilon}

\def\sgn{\mathrm{sgn}}

\makeop[ind]

\begin{document}

\title[Refined abelian Stark conjectures]{Refined abelian Stark conjectures and the equivariant leading term conjecture of Burns}
\author{Takamichi Sano}
\email{tkmc310@a2.keio.jp}
\address{Department of Mathematics\\Keio University\\3-14-1 Hiyoshi\\Kohoku-ku\\Yokohama\\223-8522\\Japan}

\thanks{The author is supported by Grant-in-Aid for JSPS Fellows.}

\begin{abstract}
We formulate a conjecture which generalizes Darmon's ``refined class number formula". We discuss relations between our conjecture and the equivariant leading term conjecture of Burns. As an application, we give another proof of 
the ``except $2$-part" of Darmon's conjecture, which was first proved by Mazur and Rubin.
\end{abstract}

\maketitle
\section{Introduction}

In \cite{B}, Burns formulated a refinement of the abelian Stark conjecture, which generalizes Gross's ``refined class number formula'' (\cite[Conjecture 4.1]{G}). He proved that a natural leading term conjecture, which is a special case of the ``equivariant Tamagawa number conjecture (ETNC)" (\cite[Conjecture 4 (iv)]{BF}) in the number field case, implies 
his refined abelian Stark conjecture (\cite[Theorem 3.1]{B}). Thus, he observed that Gross's conjecture is a consequence of the leading term conjecture. 

In this paper, using the idea of Darmon (\cite{D}), we attempt to generalize Burns's conjecture. Our main conjecture (Conjecture \ref{hnr}) is formulated as a generalization of  Darmon's ``refined class number formula" (\cite[Conjecture 4.3]{D}). We reformulate Burns's conjecture in Conjecture \ref{bconj} with slight modifications, and also propose some auxiliary conjectures (Conjectures \ref{descent} and \ref{phiconj}). We prove the following relation among these conjectures: assuming Conjecture \ref{phiconj}, Conjecture \ref{hnr} holds if and only if Conjectures \ref{descent} and \ref{bconj} hold (see Theorem \ref{rel}). Using the result of Burns (\cite[Theorem 3.1]{B}), we know that most of Conjecture \ref{bconj} is a consequence of the leading term conjecture (see Theorem \ref{bthm}). Hence, assuming Conjectures \ref{descent} and \ref{phiconj}, we deduce that Conjecture \ref{hnr} is a consequence of the leading term conjecture (see Theorem \ref{mainthm}). This is the main theorem of this paper. 

Our main theorem has the following application.  We can prove Conjectures \ref{descent} and \ref{phiconj} in the ``rank $1$" case, which was considered by Darmon, and deduce that (most of) Darmon's conjecture is a consequence of the leading term conjecture. By the works of Burns, Greither, and Flach (\cite{BG}, \cite{F}), the leading term conjecture is known to be true in this case. Hence, we give a proof of (most of) Darmon's conjecture. To be precise, we show that the ETNC for a particular Tate motive for abelian fields implies the ``except $2$-part" of Darmon's conjecture. In \cite{MR2}, Mazur and Rubin solved the ``except $2$-part" of Darmon's conjecture by using the theory of Kolyvagin systems (\cite{MR1}). Our approach gives another proof for it.

We sketch the idea of formulating Conjecture \ref{hnr}. Let $L'/L/k$ be a tower of finite extensions of global fields such that $L'/k$ is  abelian. We use Rubin's integral refinement of the abelian Stark conjecture (the Rubin-Stark conjecture, \cite{R2}). (This is Conjecture \ref{rsconj} in this paper.) Assuming the Rubin-Stark conjecture, let $\varep'$ and $\varep$ be the Rubin-Stark units lying over $L'$ and $L$ respectively. We define a ``higher norm'' $\N_{L'/L}^{(d)}(\varep')$ of $\varep'$, motivated by Darmon's construction of the ``theta-element" in \cite{D}. Roughly speaking, we observe the following property of the higher norm: we have
$$\Phi(\varep')=\Phi^{\Gal(L'/L)}(\N_{L'/L}^{(d)}(\varep'))$$
for every ``evaluator" $\Phi$ (see Proposition \ref{propnorm}). Burns's formulation (Conjecture \ref{bconj}) says that the equality 
$\Phi(\varep')=\Phi^{\Gal(L'/L)}(R(\varep))$ 
holds for every evaluator $\Phi$, where $R$ is the map constructed by local reciprocity maps. Therefore, it is natural to guess that the following equality holds: 
$$\N_{L'/L}^{(d)}(\varep')=R(\varep).$$
This equality is exactly our formulation of Conjecture \ref{hnr}, which generalizes Darmon's conjecture. 


After the author wrote the first version of this paper, the author was informed from Prof. Rubin that Mazur and Rubin also found the same conjecture as Conjecture \ref{hnr}. 
After that, their paper \cite{MR3} appeared in arXiv, and their conjecture is described in \cite[Conjecture 5.2]{MR3}. 
The author should also remark that, in the first version of this paper, there was a mistake in the formulation of Conjecture \ref{hnr}. 
We remark that the map $\mathbf{j}_{L/K}$ in \cite[Lemma 4.9]{MR3} is essentially the same as our injection $i$ in Lemma \ref{inj}, but Mazur and Rubin do not mention that $\mathbf{j}_{L/K}$ is injective. So our formulation of Conjecture \ref{hnr} is slightly stronger than \cite[Conjecture 5.2]{MR3}.

The organization of this paper is as follows. In \S \ref{alg}, we give algebraic foundations which will be frequently used in the subsequent sections. In \S \ref{secconj}, 
after a short preliminary on the Rubin-Stark conjecture and a review of some related known facts, we formulate the main conjectures, and also 
prove the main theorem (Theorem \ref{mainthm}). In \S \ref{appl}, as an application of Theorem \ref{mainthm}, we give another proof of 
the ``except $2$-part" of Darmon's conjecture (Mazur-Rubin's theorem).

\begin{notation}

For any abelian group $G$, $\bbZ[G]$-modules are simply called $G$-modules. The tensor product over $\bbZ[G]$ is denoted by $$-\otimes_G-.$$
Similarly, the exterior power over $\bbZ[G]$, and $\Hom$ of $\bbZ[G]$-modules are denoted by 
$$\bigwedge_G \ , \ \Hom_G(-,-)$$
respectively. We use the notations like this also for $\bbZ[G]$-algebras. 

For any subgroup $H$ of $G$, we define the norm element $\N_H\in \bbZ[G]$ by 
$$\N_H=\sum_{\sigma \in H}\sigma.$$

For any $G$-module $M$, we define 
$$M^G=\{ m \in M \ | \ \sigma m=m \mbox{ for all }\sigma \in G \}.$$

The maximal $\bbZ$-torsion subgroup of $M$ is denoted by $M_{\tors}$. 

For any $G$-modules $M$ and $M'$, we endow $M\otimes_\bbZ M'$ with a structure of a $G$-bimodule by 
$$\sigma(m\otimes m')=\sigma m \otimes m' \quad \text{and}\quad  \ (m \otimes m')\sigma= m\otimes \sigma m',$$
where $\sigma \in G$, $m \in M$ and $m' \in M'$. If $\varphi \in \Hom_G(M,M'')$, where $M''$ is another $G$-module, 
we often denote $\varphi \otimes \Id \in \Hom_G(M\otimes_\bbZ M',M''\otimes_\bbZ M')$ by $\varphi$.

\end{notation}

\section{Algebra} \label{alg}

\subsection{Exterior powers} \label{ext}

Let $G$ be a finite abelian group. For a $G$-module $M$ and $\varphi\in\Hom_G(M,\bbZ[G])$, there is a $G$-homomorphism
$$\bigwedge_G^rM \longrightarrow \bigwedge_G^{r-1}M$$
for all $r\in\bbZ_{\geq1}$, defined by 
$$m_1\wedge\cdots\wedge m_r \mapsto \sum_{i=1}^r (-1)^{i-1}\varphi(m_i)m_1\wedge\cdots\wedge m_{i-1}\wedge m_{i+1}\wedge\cdots \wedge m_r.$$
This morphism is also denoted by $\varphi$.

This construction gives a morphism
\begin{eqnarray}
\bigwedge_G^s\Hom_G(M,\bbZ[G]) \longrightarrow \Hom_G(\bigwedge_G^rM, \bigwedge_G^{r-s}M) \label{extmap}
\end{eqnarray}
for all $r, s\in \bbZ_{\geq0}$ such that $r\geq s$, defined by
$$\varphi_1\wedge\cdots\wedge \varphi_s \mapsto (m \mapsto \varphi_s \circ \cdots \circ \varphi_1(m)).$$
From this, we often regard an element of $\bigwedge_G^s\Hom_G(M,\bbZ[G])$ as an element of $\Hom_G(\bigwedge_G^rM, \bigwedge_G^{r-s}M)$.
Note that if $r=s$, $\varphi_1\wedge\cdots\wedge \varphi_r\in\bigwedge_G^r\Hom_G(M,\bbZ[G])$, and $m_1\wedge \cdots \wedge m_r \in \bigwedge_G^rM$, then we have 
$$(\varphi_1\wedge \cdots \wedge \varphi_r)(m_1\wedge \cdots \wedge m_r)= \det(\varphi_i(m_j))_{1\leq i, j \leq r}.$$

For a $G$-algebra $Q$ and $\varphi \in \Hom_G(M,Q)$, there is a $G$-homomorphism
$$\bigwedge_G^rM \longrightarrow (\bigwedge_G^{r-1}M) \otimes_{G}Q$$
defined by 
$$m_1\wedge\cdots\wedge m_r \mapsto \sum_{i=1}^r (-1)^{i-1}m_1\wedge\cdots\wedge m_{i-1}\wedge m_{i+1}\wedge\cdots \wedge m_r \otimes \varphi(m_i).$$
Similarly to the construction of (\ref{extmap}), we have a morphism
\begin{eqnarray}
\bigwedge_G^s\Hom_G(M,Q) \longrightarrow \Hom_G(\bigwedge_G^rM, (\bigwedge_G^{r-s}M)\otimes_G Q). \label{extmap2}
\end{eqnarray}

\subsection{Rubin's lattice}
In this subsection, we fix a finite abelian group $G$ and its subgroup $H$. Following Rubin \cite[\S 1.2]{R2}, we give the following definition.

\begin{definition} \label{deflat}
For a finitely generated $G$-module $M$ and $r\in\bbZ_{\geq0}$, we define Rubin's lattice by
$$\bigcap_G^r M =\{ m\in ( \bigwedge_G^rM ) \otimes_\bbZ \bbQ \ | \ \Phi(m)\in\bbZ[G] \mbox{ for all }\Phi\in\bigwedge_G^r\Hom_G(M, \bbZ[G]) \}.$$
Note that $\bigcap_G^0 M=\bbZ[G].$

\end{definition}

\begin{remark} \label{rlat}
We define 
$\iota : \bigwedge_G^r\Hom_G(M,\bbZ[G]) \rightarrow \Hom_G(\bigwedge_G^rM,\bbZ[G])$
by $\varphi_1 \wedge \cdots \wedge \varphi_r \mapsto \varphi_r \circ \cdots \circ \varphi_1$ (see (\ref{extmap})). It is not difficult to see that 
$$\bigcap_G^r M \stackrel{\sim}{\longrightarrow} \Hom_G(\Im \iota, \bbZ[G]) \quad ; \quad m \mapsto (\Phi \mapsto \Phi(m))$$
is an isomorphism (see \cite[\S 1.2]{R2}). 
\end{remark}

\begin{remark}
If $M \rightarrow M'$ is a morphism between finitely generated $G$-modules, then it induces a natural $G$-homomorphism 
$$\bigcap_G^rM \longrightarrow \bigcap_G^rM'.$$
\end{remark}

Next, we study some more properties of Rubin's lattice.

Let $I_H$ (resp. $I(H)$) be the kernel of the natural map $\bbZ[G] \rightarrow \bbZ[G/H]$ 
(resp. $\bbZ[H] \rightarrow \bbZ$). 
Note that $I(H) \subset I_H$.
For any $d\in\bbZ_{\geq0}$, let $Q_H^d$ (resp. $Q(H)^d$) be the $d$-th augmentation quotient $I_H^d/I_H^{d+1}$ (resp. $I(H)^d/I(H)^{d+1}$). Note that $Q_H^d$ has a natural $G/H$-module structure, since $\bbZ[G]/I_H \simeq \bbZ[G/H]$. It is known that there is a natural isomorphism of $G/H$-modules 
\begin{eqnarray}
\bbZ[G/H] \otimes_\bbZ Q(H)^d \stackrel{\sim}{\longrightarrow} Q_H^d \label{eqaug}
\end{eqnarray}
given by
$$\sigma\otimes \bar a \mapsto \overline{\widetilde \sigma a},$$
where $a \in I(H)^d$ and $\bar a$ denote the image of $a$ in $Q(H)^d$, $\widetilde \sigma \in G$ is any lift of $\sigma \in G/H$, and $\overline{\widetilde \sigma a}$ denote the image of $\widetilde \sigma a \in I_H^d$ in $Q_H^d$ ($\overline{\widetilde \sigma a}$ does not depend on the choice of $\widetilde \sigma$) (see \cite[Lemma 5.2.3 (2)]{P}). We often identify $\bbZ[G/H] \otimes_\bbZ Q(H)^d$ and $Q_H^d$.

The following lemma is well-known, and we omit the proof.

\begin{lemma} \label{hom}
For a $G$-module $M$ and an abelian group $A$, there is a natural isomorphism
$$\Hom_\bbZ(M,A) \stackrel{\sim}{\longrightarrow} \Hom_G(M,\bbZ[G]\otimes_\bbZ A) \quad ; \quad \varphi \mapsto (m \mapsto \sum_{\sigma \in G} \sigma^{-1} \otimes \varphi(\sigma m)).$$
\end{lemma}

\begin{lemma} \label{homin}
Let $M$ be a finitely generated $G/H$-module, and $\overline M=M/M_{\rm{tors}}$. For any $d\in\bbZ_{\geq0}$, we have an isomorphism
$$\Hom_{G/H}(M,\bbZ[G/H])\otimes_\bbZ Q(H)^d \stackrel{\sim}{\longrightarrow} \Hom_{G/H}(\overline M,Q_H^d) \quad ; \quad \varphi\otimes a \mapsto (\bar m \mapsto \varphi(m)a).$$
In particular, 
$$\Hom_{G/H}(M,\bbZ[G/H])\otimes_\bbZ Q(H)^d \longrightarrow \Hom_{G/H}( M,Q_H^d) $$
is an injection.
\end{lemma}

\begin{proof}
We have a commutative diagram:
\[\xymatrix{
\Hom_{G/H}(M,\bbZ[G/H]) \otimes_\bbZ Q(H)^d \ar[d] \ar[r] & \Hom_{G/H}(\overline M,Q_H^d)  \ar[d] \\
\Hom_\bbZ(M,\bbZ) \otimes_\bbZ Q(H)^d  \ar[r] & \Hom_\bbZ(\overline M,Q(H)^d), \\
}\] 
where the bottom horizontal arrow is given by $\varphi \otimes a \mapsto (\bar m \mapsto \varphi(m)a)$, and the left and right vertical arrows are the isomorphisms given in Lemma \ref{hom} (note that we have a natural isomorphism $Q_H^d \simeq \bbZ[G/H]\otimes_\bbZ Q(H)^d$, see (\ref{eqaug})). The bottom horizontal arrow is an isomorphism, since $\Hom_\bbZ(M,\bbZ) \simeq \Hom_\bbZ(\overline M,\bbZ)$ and $\overline M$ is torsion-free by definition. Hence the upper horizontal arrow is also bijective.
\end{proof}


\begin{definition}
A finitely generated $G$-module $M$ is called a $G$-lattice if $M$ is torsion-free.
\end{definition}

For example, for a finitely generated $G$-module $M$, $\Hom_G(M,\bbZ[G])$ is a $G$-lattice. Rubin's lattice $\bigcap_G^rM$ is also a $G$-lattice.

\begin{proposition} \label{propphi}
Let $M$ be a $G/H$-lattice, and $r, d \in \bbZ_{\geq0}$ such that $r\geq d$. Then an element $\Phi \in \bigwedge_{G/H}^d\Hom_{G/H}(M, Q_H^1)$ induces a $G/H$-homomorphism
$$\bigcap_{G/H}^rM \longrightarrow (\bigcap_{G/H}^{r-d}M)\otimes_{G/H}Q_H^d(\simeq (\bigcap_{G/H}^{r-d}M)\otimes_\bbZ Q(H)^d).$$
\end{proposition}

\begin{proof}
Note that $Q_H^1$ is the degree-$1$-part of the graded $G/H$-algebra $\bigoplus_{i\geq0}Q_H^i$. We apply (\ref{extmap2}) to know that $\Phi$ induces the $G/H$-homomorphism
\begin{eqnarray}
\bigwedge_{G/H}^rM \longrightarrow (\bigwedge_{G/H}^{r-d}M)\otimes_{G/H}Q_H^d. \label{phimap}
\end{eqnarray}
We extend this map to Rubin's lattice $\bigcap_{G/H}^rM$. We may assume that there exist $\varphi_1,\ldots, \varphi_d \in \Hom_{G/H}(M,Q_H^1)$ such that $\Phi=\varphi_1 \wedge \cdots \wedge \varphi_d$. Moreover, by Lemma \ref{homin}, we may assume for each $1 \leq i \leq d$ that there exist $\psi_i \in \Hom_{G/H}(M,\bbZ[G/H])$ and $a_i \in Q(H)^1$ such that $\varphi_i=\psi_i(\cdot)a_i$. Put $\Psi=\psi_1 \wedge \cdots \wedge \psi_d \in \bigwedge_{G/H}^d \Hom_{G/H}(M,\bbZ[G/H])$. By the definition of Rubin's lattice, $\Phi$ induces a $G/H$-homomorphism
$$\bigcap_{G/H}^r M \longrightarrow (\bigcap_{G/H}^{r-d}M) \otimes_\bbZ Q(H)^d \quad ; \quad m \mapsto \Psi(m) \otimes a_1 \cdots a_d.$$
This extends the map (\ref{phimap}).
\end{proof} 

The following definition is due to \cite[\S 2.1]{B}.

\begin{definition}
Let $M$ be a $G$-lattice. For $\varphi \in \Hom_G(M,\bbZ[G])$, we define $\varphi^H \in \Hom_{G/H}(M^H,\bbZ[G/H])$ by 
$$M^H \stackrel{\varphi}{\longrightarrow} \bbZ[G]^H \stackrel{\sim}{\longrightarrow} \bbZ[G/H],$$
where the last isomorphism is given by $\N_H \mapsto 1$. Similarly, for $\Phi \in \bigwedge_G^r\Hom_G(M,\bbZ[G])$ ($r\in\bbZ_{\geq0}$), $\Phi^H \in \bigwedge_{G/H}^r\Hom_{G/H}(M^H,\bbZ[G/H])$ is defined. (If $r=0$, we define $\Phi^H \in \bbZ[G/H]$ to be the image of $\Phi \in \bbZ[G]$ under the natural map.)
\end{definition}

\begin{remark} \label{Hrem}
It is easy to see that 
$$\varphi^H =\sum_{\sigma \in G/H}\varphi^1(\sigma ( \ \cdot \ ))\sigma^{-1},$$
where $\varphi^1 \in \Hom_\bbZ(M,\bbZ)$ corresponds to $\varphi \in \Hom_G(M,\bbZ[G])$ (see Lemma \ref{hom}). If $r \geq 1$, then one also sees that 
\begin{eqnarray}
\Phi(m)=\Phi^H(\N_H^r m) \quad \text{in} \quad \bbZ[G/H] \label{eqphi}
\end{eqnarray}
for all $\Phi \in \bigwedge_G^r \Hom_G(M,\bbZ[G])$ and $m \in \bigcap_G^r M$. 
\end{remark}

\begin{lemma} \label{surj}
If $M$ is a $G$-lattice, then the map
$$\Hom_G(M,\bbZ[G]) \longrightarrow \Hom_{G/H}(M^H,\bbZ[G/H]) \quad ; \quad \varphi \mapsto \varphi^H$$
is surjective.
\end{lemma}

\begin{proof}
By Remark \ref{Hrem}, what we have to prove is that the restriction map
$$\Hom_\bbZ(M,\bbZ) \longrightarrow \Hom_\bbZ(M^H,\bbZ)$$
is surjective. Therefore, it is sufficient to prove that $M/M^H$ is torsion-free. Take $m\in M$ such that $nm \in M^H$ for a nonzero $n\in\bbZ$. For any $\sigma\in H$, we have
$$n((\sigma-1)m)=(\sigma-1)nm=0.$$
Since $M$ is a $G$-lattice, it is torsion-free. Therefore, we have $(\sigma-1)m=0$. This implies $m \in M^H$.
\end{proof}

\begin{lemma} \label{inj}
Let $M$ be a $G$-lattice, and $r, d \in \bbZ_{\geq0}$. Then there is a canonical injection
$$i: \bigcap_{G/H}^r M^H \longrightarrow \bigcap_{G}^r M.$$
Furthermore, the maps
$$(\bigcap_{G/H}^rM^H) \otimes_\bbZ Q(H)^d \stackrel{i}{\longrightarrow}( \bigcap_G^rM )\otimes_\bbZ Q(H)^d \longrightarrow (\bigcap_G^rM)\otimes_\bbZ \bbZ[H]/I(H)^{d+1}$$
are both injective, where the first arrow is induced by $i$, and the second by the inclusion $Q(H)^d \hookrightarrow \bbZ[H]/I(H)^{d+1}$.
\end{lemma}

\begin{proof}
Let 
$$\iota : \bigwedge_G^r\Hom_G(M,\bbZ[G]) \longrightarrow \Hom_G(\bigwedge_G^rM,\bbZ[G])$$
and 
$$\iota_H : \bigwedge_{G/H}^r\Hom_{G/H}(M^H,\bbZ[G/H]) \longrightarrow \Hom_{G/H}(\bigwedge_{G/H}^rM^H,\bbZ[G/H])$$
be the maps in Remark \ref{rlat}. It is easy to see that the map 
$$\kappa : \Im\iota \longrightarrow \Im\iota_H \quad ; \quad \iota(\Phi) \mapsto \iota_H(\Phi^H) $$
is well-defined. By Lemma \ref{surj}, the map 
$$\bigwedge_G^r\Hom_G(M,\bbZ[G]) \longrightarrow \bigwedge_{G/H}^r\Hom_{G/H}(M^H,\bbZ[G/H]) \quad ; \quad \Phi \mapsto \Phi^H$$
is surjective. So the map $\kappa$ is also surjective. Hence, by Remark \ref{rlat}, we have an injection 
$$i:\bigcap_{G/H}^rM^H \longrightarrow \bigcap_{G}^rM$$
(note that $\Hom_{G/H}(\Im \iota_H, \bbZ[G/H]) \simeq \Hom_{G}(\Im \iota_H,\bbZ[G])$ by Lemma \ref{hom}). 
The cokernel of this map is isomorphic to a submodule of $\Hom_G(\Ker \kappa,\bbZ[G])$, so it is torsion-free. Hence the map 
$$i:(\bigcap_{G/H}^rM^H) \otimes_\bbZ Q(H)^d \longrightarrow( \bigcap_G^rM) \otimes_\bbZ Q(H)^d$$
is injective. The injectivity of the map
$$(\bigcap_G^rM )\otimes_\bbZ Q(H)^d \longrightarrow (\bigcap_G^rM)\otimes_\bbZ \bbZ[H]/I(H)^{d+1}$$
follows from the fact that $\bigcap_G^rM$ is torsion-free.
\end{proof}

\begin{remark} \label{reminj}
The canonical injection $i:\bigcap_{G/H}^rM^H\hookrightarrow\bigcap_G^rM$ constructed above does not coincide in general with the map induced by the inclusion $M^H\hookrightarrow M$. In fact, if $r \geq 1$, then we have 
$$i(\N_H^r m)=\N_H m$$
for all $m \in \bigcap_G^r M$. 
\end{remark}


\begin{definition} \label{defnorm}
Let $M$ be a $G$-lattice, and $r,d\in\bbZ_{\geq0}$. When $r\geq1$, we define the $d$-th norm 
$$\N_H^{(r,d)} : \bigcap_G^rM \longrightarrow( \bigcap_G^rM) \otimes_\bbZ\bbZ[H]/I(H)^{d+1}$$
by 
$$\N_H^{(r,d)}(m)=\sum_{\sigma \in H}\sigma m\otimes \sigma^{-1}.$$
When $r=0$, we define 
$$\N_H^{(0,d)} : \bbZ[G]\longrightarrow \bbZ[G]/I_H^{d+1}$$
to be the natural map.
\end{definition}

\begin{remark} \label{remnorm}
The $0$-th norm is the usual norm :
$$\N_H^{(r,0)}=
\begin{cases}
\N_H &\text{if $r\geq1$,} \\
\bbZ[G] \longrightarrow \bbZ[G/H] &\text{if $r=0$.}
\end{cases}$$
\end{remark}

\begin{proposition} \label{propnorm}
Let $M$ be a $G$-lattice, $r, d \in \bbZ_{\geq 0}$, and $m \in \bigcap_G^r M$. Assume 
$$\N_H^{(r,d)}(m) \in \Im i,$$
where, in the case $r\geq1$, $i: (\bigcap_{G/H}^r M^H) \otimes_\bbZ Q(H)^d \to (\bigcap_G^rM)\otimes_\bbZ \bbZ[H]/I(H)^{d+1}$ is defined to be the injection in Lemma \ref{inj}, and in the case $r=0$, $i: Q_H^d \hookrightarrow \bbZ[G]/I_H^{d+1}$ to be the inclusion. If $d=0$ or $r=0$ or $1$, then we have 
$$\Phi(m)=\Phi^H(i^{-1}(\N_H^{(r,d)}(m))) \quad \text{in} \quad Q_H^d$$
for all $\Phi \in \bigwedge_G^r\Hom_G(M,\bbZ[G])$.
\end{proposition}

\begin{proof}
When $d=0$, the proposition follows from Remarks \ref{Hrem}, \ref{reminj}, and \ref{remnorm}. When $r=0$, the proposition is clear. So we suppose $r=1$. Note that in this case the map $i$ is the inclusion 
$$i: M^H \otimes_\bbZ Q(H)^d \hookrightarrow M \otimes_\bbZ \bbZ[H]/I(H)^{d+1}.$$
We regard $M^H \otimes_\bbZ Q(H)^d \subset M \otimes_\bbZ \bbZ[H]/I(H)^{d+1}.$ 

Take any $\varphi \in \Hom_G(M,\bbZ[G])$. Then $\varphi^H$ is written by
$$\varphi^H=\sum_{\sigma\in G/H}\varphi^1(\sigma(  \cdot  ))\sigma^{-1}$$
(see Remark \ref{Hrem}). For each $\sigma \in G/H$, we fix a lifting $\widetilde \sigma \in G$, and put 
$$\widetilde \varphi =\sum_{\sigma\in G/H}\varphi^1(\widetilde \sigma(  \cdot  )){\widetilde \sigma}^{-1} \in \Hom_\bbZ(M,\bbZ[G]).$$
Then, by the assumption on $\N_H^{(1,d)}(m)$, we have 
$$ \varphi^H(\N_H^{(1,d)}(m))=(\alpha \circ(\widetilde \varphi\otimes \Id))(\N_H^{(1,d)}(m)) \in Q_H^d, $$
where 
$$\alpha : \bbZ[G]\otimes_\bbZ \bbZ[H]/I(H)^{d+1} \longrightarrow \bbZ[G]/I_H^{d+1} \quad;\quad a\otimes \overline b \mapsto \overline{ab}.$$
It is easy to check that 
$$\varphi(m)=(\alpha \circ(\widetilde \varphi\otimes \Id))(\N_H^{(1,d)}(m)) \quad \text{in} \quad \bbZ[G]/I_H^{d+1}. $$ 
(This can be checked by noting that 
$$\varphi=\sum_{\sigma \in G/H}\sum_{\tau \in H}\varphi^1(\widetilde \sigma \tau (\cdot))\widetilde \sigma^{-1}\tau^{-1}.)$$
Hence we have 
$$\varphi(m)=\varphi^H(\N_H^{(1,d)}(m)) \quad \text{in} \quad Q_H^d.$$
\end{proof}

\begin{remark}
We expect that the assertion in Proposition \ref{propnorm} holds for general $r$ and $d$. (See Conjecture \ref{phiconj} in \S \ref{refconj}.)
\end{remark}

\begin{theorem} \label{thminj}
Let $M$ be a $G$-lattice, and $r,d \in \bbZ_{\geq0}$. Then the map 
$$(\bigcap_{G/H}^rM^H)\otimes_\bbZ Q(H)^d \longrightarrow \Hom_G(\bigwedge_G^r\Hom_G(M,\bbZ[G]),Q_H^d) \quad ; \quad \alpha \mapsto (\Phi \mapsto \Phi^H(\alpha))$$
is injective.
\end{theorem}

\begin{proof}
Let
$$\iota_H : \bigwedge_{G/H}^r\Hom_{G/H}(M^H,\bbZ[G/H]) \longrightarrow \Hom_{G/H}(\bigwedge_{G/H}^rM^H,\bbZ[G/H]).$$
be the map defined in Remark \ref{rlat} for $G/H$ and $M^H$. Taking $\Hom_{G/H}(-,\bbZ[G/H])$ to the exact sequence
$$0 \longrightarrow \Ker \iota_H \longrightarrow \bigwedge_{G/H}^r\Hom_{G/H}(M^H,\bbZ[G/H]) \longrightarrow \Im \iota_H \longrightarrow 0,$$
we have the exact sequence
$$0 \longrightarrow \bigcap_{G/H}^rM^H \longrightarrow \Hom_{G/H}(\bigwedge_{G/H}^r\Hom_{G/H}(M^H,\bbZ[G/H]),\bbZ[G/H]) \longrightarrow \Hom_{G/H}(\Ker \iota_H,\bbZ[G/H]).$$
Since $\Hom_{G/H}(\Ker \iota_H,\bbZ[G/H])$ is torsion-free, the map 
\begin{eqnarray}
{(\bigcap_{G/H}^r M^H) \otimes_\bbZ Q(H)^d \longrightarrow \Hom_{G/H}(\bigwedge_{G/H}^r\Hom_{G/H}(M^H,\bbZ[G/H]),\bbZ[G/H])\otimes_\bbZ Q(H)^d }   \nonumber   
\end{eqnarray}
is injective. From Lemma \ref{homin}, we have an injection 
\begin{eqnarray}
&& \Hom_{G/H}(\bigwedge_{G/H}^r\Hom_{G/H}(M^H,\bbZ[G/H]),\bbZ[G/H])\otimes_\bbZ Q(H)^d \nonumber \\
& \longrightarrow &  \Hom_{G/H}(\bigwedge_{G/H}^r\Hom_{G/H}(M^H,\bbZ[G/H]),Q_H^d) 
=\Hom_{G}(\bigwedge_{G/H}^r\Hom_{G/H}(M^H,\bbZ[G/H]),Q_H^d). \nonumber
\end{eqnarray}
From Lemma \ref{surj}, we also have an injection 
\begin{eqnarray}
\Hom_{G}(\bigwedge_{G/H}^r\Hom_{G/H}(M^H,\bbZ[G/H]),Q_H^d) \longrightarrow \Hom_G(\bigwedge_G^r \Hom_G(M,\bbZ[G]),Q_H^d). \nonumber
\end{eqnarray}
The composition of the above three injections coincides with the map given in the theorem, hence we complete the proof.
\end{proof}

\section{Conjectures} \label{secconj}
\subsection{Notation} \label{not}
Throughout this section, we fix a global field $k$. We also fix $T$, a finite set of places of $k$, containing no infinite place.
For a finite separable extension $L/k$ and a finite set $S$ of places of $k$, $S_L$ denotes the set of places of $L$ lying above the places in $S$. For $S$ containing all the infinite places and disjoint to $T$, $\cO_{L,S,T}^\times$ denotes the $(S,T)$-unit group of $L$, i.e. 
$$\cO_{L,S,T}^\times = \{ a\in L^\times \ | \ \ord_w(a)=0 \mbox{ for all } w\notin S_L \mbox{ and } a \equiv 1 \ (\mod \ w') \mbox{ for all } w'\in T_L \},$$
where $\ord_w$ is the (normalized) additive valuation at $w$. Let $Y_{L,S}=\bigoplus_{w\in S_L}\bbZ w$, the free abelian group on $S_L$, and $X_{L,S}=\{ \sum a_ww \in Y_{L,S} \ | \ \sum a_w=0 \}$. Let 
$$\lambda_{L,S} : \cO_{L,S,T}^\times \longrightarrow \bbR \otimes_\bbZ X_{L,S}$$
be the map defined by $\lambda_{L,S}(a)=-\sum_{w\in S_L}\log|a|_w w$, where $| \cdot |_w$ is the normalized absolute value at $w$.

Let $\Omega(=\Omega(k,T))$ be the set of triples $(L,S,V)$ satisfying the following:
\begin{itemize}
\item{$L$ is a finite abelian extension of $k$,}
\item{$S$ is a nonempty finite set of places of $k$ satisfying
\begin{itemize}
\item{$S \cap T=\emptyset$,}
\item{$S$ contains all the infinite places and all places ramifying in $L$,}
\item{$\cO_{L,S,T}^\times$ is torsion-free,} 
\end{itemize}}
\item{$V$ is a subset of $S$ satisfying
\begin{itemize}
\item{any $v\in V$ splits completely in $L$,}
\item{$|S| \geq |V|+1$.}
\end{itemize}}
\end{itemize}
We assume that $\Omega \neq \emptyset$. If $k$ is a number field, then the condition that $\cO_{L,S,T}^\times$ is torsion-free is satisfied when, for example, $T$ contains two finite places of unequal residue characteristics. 

Take $(L,S,V) \in \Omega$, and put $\cG_L=\Gal(L/k)$, $r=r_V=|V|$. The equivariant $L$-function attached to the data $(L/k,S,T)$ is defined by 
$$\Theta_{L,S,T}(s) = \sum_{\chi \in \widehat \cG_L}e_\chi L_{S,T}(s,\chi^{-1}),$$
where $\widehat \cG_L=\Hom_\bbZ(\cG_L,\bbC^\times)$, $e_\chi=\frac1{|\cG_L|}\sum_{\sigma\in \cG_L}\chi(\sigma)\sigma^{-1}$, and 
$$L_{S,T}(s,\chi)=\prod_{v\in T}(1-\chi(\Fr_v)\N v^{1-s})\prod_{v\notin S}(1-\chi(\Fr_v)\N v^{-s})^{-1},$$
where $\Fr_v \in \cG_L$ is the arithmetic Frobenius at $v$, and $\N v$ is the cardinality of the residue field at $v$. 

We define 
$$\Lambda_{L,S,T}^r=\{ a\in \bigcap_{\cG_L}^r\cO_{L,S,T}^\times  \ | \ e_\chi a=0 \mbox{ for every } \chi \in \widehat \cG_L \mbox{ such that } r(\chi) >r \},$$
where $r(\chi)=r(\chi,S)=\ord_{s=0}L_{S,T}(s,\chi)$ (for the definition of $\bigcap_{\cG_L}^r$, see Definition \ref {deflat}). It is well-known that 
$$r(\chi)=
\begin{cases}
|\{ v\in S \ | \ v \mbox{ splits completely in } L^{\Ker \chi} \}| &\text{if $\chi$ is nontrivial}, \\
|S|-1 &\text{if $\chi$ is trivial},
\end{cases}   $$
(see \cite[Proposition 3.4, Chpt. I]{T}) so by our assumptions on $V$, we have $r(\chi) \geq r$ for every $\chi$. This implies that  $s^{-r}\Theta_{L,S,T}(s)$ is holomorphic at $s=0$. We define 
$$\Theta_{L,S,T}^{(r)}(0)=\lim_{s\rightarrow 0}s^{-r}\Theta_{L,S,T}(s) \in \bbC[\cG_L].$$

We fix the following:
\begin{itemize}
\item{a bijection $\{ \mbox{all the places of } k \} \simeq \bbZ_{\geq 0}$,}
\item{for each place $v$ of $k$, a place of $\bar k$ (a fixed separable closure of $k$) lying above $v$.}
\end{itemize}
From this fixed choice, we can regard $V$ as a totally ordered finite set with order $\prec$, and arrange $V=\{ v_1,\ldots, v_r \}$ so that $v_1 \prec \cdots \prec v_r$. 
For each $v \in V$, there is a fixed place $w$ of $L$ lying above $v$, and define $v^\ast \in \Hom_{\cG_L}(Y_{L,S},\bbZ[\cG_L])$ to be the dual of $w$, i.e. 
$$v^\ast(w')= \sum_{\sigma w=w'}\sigma.$$
Thus, we often use slightly ambiguous notations such as follows: the fixed places of $L$ lying above $v, v', v_i$, etc. are denoted by $w, w', w_i$, etc. respectively. 
We define the analytic regulator map $R_V : \bigwedge_{\cG_L}^r \cO_{L,S,T}^\times \rightarrow \bbR[\cG_L]$ by 
$$R_V=\bigwedge_{v\in V}(v^\ast \circ \lambda_{L,S}),$$
where the exterior power in the right hand side means $(v_1^\ast \circ \lambda_{L,S} \wedge \cdots \wedge v_r^\ast \circ \lambda_{L,S})$ (defined similarly to (\ref{extmap})). Thus, when we take an 
exterior power on a totally ordered finite set, we always mean that the order is arranged to be ascending order. One can easily see that 
$$v^\ast \circ \lambda_{L,S}=-\sum_{\sigma \in \cG_{L}}\log|\sigma(\cdot)|_w \sigma^{-1},$$
so a more explicit definition of $R_V$ is as follows: 
$$R_V(u_1\wedge \cdots \wedge u_r)=\det(-\sum_{\sigma \in \cG_L}\log|\sigma(u_i)|_{w_j}\sigma^{-1}).$$

\subsection{The Rubin-Stark conjecture} \label{secrsconj}
We use the notations and conventions as in \S \ref{not}. Recall that the integral refinement of abelian Stark conjecture, which we call Rubin-Stark conjecture, formulated by Rubin, is stated as follows:

\begin{conjecture}[Rubin {\cite[Conjecture B$'$]{R2}}] \label{rsconj}
For $(L,S,V) \in \Omega$, there is a unique $\varep_{L,S,V}=\varep_{L,S,T,V}\in \Lambda_{L,S,T}^r$ such that 
$$R_V(\varep_{L,S,V})=\Theta_{L,S,T}^{(r)}(0).$$
\end{conjecture}

The element $\varep_{L,S,V}$ predicted by the conjecture is called Rubin-Stark unit, Rubin-Stark element, or simply Stark unit, etc. 
In this paper we call it Rubin-Stark unit. 

\begin{remark} \label{rzero}
When $r=0$, Conjecture \ref{rsconj} is known to be true (see \cite[Theorem 3.3]{R2}). In this case we have $\varep_{L,S,V}=\Theta_{L,S,T}(0) \in \bbZ[\cG_L]=\bigcap_{\cG_L}^0\cO_{L,S,T}^\times$.
\end{remark}

\begin{remark} \label{rstriv}
When $r <  \min \{ |S|-1, |\{ v\in S \ | \ v \mbox{ splits completely in }L \}| \}$, we have $\Theta_{L,S,T}^{(r)}(0)=0$, so Conjecture \ref{rsconj} is trivially true (namely, we have $\varep_{L,S,V}=0$).
\end{remark}

\begin{remark}
When $k=\bbQ$, Conjecture \ref{rsconj} is true for any $T$ and $(L,S,V)\in \Omega(\bbQ,T)$ (see \cite[Theorem A]{B}).
\end{remark}
\subsection{Some properties of Rubin-Stark units}
In this subsection, we assume that Conjecture \ref{rsconj} holds for all $(L,S,V) \in \Omega$, and review some properties of Rubin-Stark units. 

\begin{lemma}[{\cite[Lemma 2.7 (ii)]{R2}}] \label{reginj}
Let $(L,S,V) \in \Omega$. Then $R_V$ is injective on $\bbQ \otimes_\bbZ \Lambda_{L,S,T}^r$.
\end{lemma}

\begin{proof}
Since $\lambda_{L,S}$ induces an injection $\bbQ \otimes_\bbZ \bigwedge_{\cG_L}^r \cO_{L,S,T}^\times \rightarrow \bbC \otimes_\bbZ \bigwedge_{\cG_L}^r X_{L,S}$, it is sufficient to prove that 
$$\bigwedge_{v\in V}v^\ast :  e_\chi(\bbC \otimes_\bbZ\bigwedge_{\cG_L}^rX_{L,S}) \longrightarrow \bbC[\cG_L]$$
is injective for every $\chi \in \widehat \cG_L$ such that $r(\chi)=r$. It is well-known that $r(\chi)=\dim_\bbC(e_\chi(\bbC \otimes_\bbZ X_{L,S}))$, so we have $\dim_\bbC(e_\chi(\bbC \otimes_\bbZ\bigwedge_{\cG_L}^rX_{L,S}))=1$. Take any $v' \in S \setminus V$, then we have 
$$(\bigwedge_{v\in V}v^\ast) (e_\chi \bigwedge_{v\in V}(w-w'))=e_\chi \neq 0,$$
(recall that $w$ (resp. $w'$) denotes the fixed place of $L$ lying above $v$ (resp. $v'$)), which proves the lemma.
\end{proof}

\begin{proposition}[{\cite[Proposition 6.1]{R2}}] \label{nr}
Let $(L,S,V),(L',S',V) \in \Omega$, and suppose that $L\subset L'$ and $S\subset S'$. Then we have 
$$ \N_{L'/L}^r(\varep_{L',S',V})=(\prod_{v \in S'\setminus S}(1-\Fr_v^{-1}))\varep_{L,S,V},$$
where $\N_{L'/L}=\N_{\Gal(L'/L)}$, and if $r=0$, then we regard $\N_{L'/L}^r$ as the natural map $\bbZ[\cG_{L'}]\to\bbZ[\cG_L]$.
\end{proposition}

\begin{proof}
It is easy to see that $\N_{L'/L}^r(\varep_{L',S',V}) \in \bbQ \otimes_\bbZ \Lambda_{L,S',T}^r$. Hence, by Lemma \ref{reginj}, it is enough to check that 
$$R_V(\N_{L'/L}^r(\varep_{L',S',V}))=R_V((\prod_{v \in S'\setminus S}(1-\Fr_v^{-1}))\varep_{L,S,V}).$$
The left hand side is equal to the image of $\Theta_{L',S',T}^{(r)}(0)$ in $\bbR[\cG_L]$, and hence to $\prod_{v \in S' \setminus S}(1-\Fr_v^{-1})\Theta_{L,S,T}^{(r)}(0)$ (see \cite[Proposition 1.8, Chpt. IV]{T}). The right hand side is equal to $\prod_{v \in S' \setminus S}(1-\Fr_v^{-1})\Theta_{L,S,T}^{(r)}(0)$, so we complete the proof.
\end{proof}

\begin{proposition}[{\cite[Lemma 5.1 (iv) and Proposition 5.2]{R2}}] \label{ordr}
Let $(L,S,V), (L,S',V') \in \Omega$, and suppose that $S \subset S'$, $V\subset V'$ and $S'\setminus S=V'\setminus V$. Put 
$$\Phi_{V',V}=\sgn(V',V)\bigwedge_{v\in V'\setminus V}(\sum_{\sigma\in \cG_L}\ord_w(\sigma(\cdot))\sigma^{-1})\in \bigwedge_{\cG_L}^{r'-r}\Hom_{\cG_L}(\cO_{L,S',T}^\times, \bbZ[\cG_L]),$$
where $r =|V|$, $r'=|V'|$, and $\sgn(V',V)=\pm 1$ is defined by 
$$(\bigwedge_{v\in V}v^\ast)\circ(\bigwedge_{v\in V'\setminus V}v^\ast)=\sgn(V',V)\bigwedge_{v\in V'}v^\ast \quad \mbox{in} \quad \Hom_{\cG_L}(\bigwedge_{\cG_L}^{r'}Y_{L,S'}, \bbZ[\cG_L]).$$
Then we have 
$$\Phi_{V',V}(\Lambda_{L,S',T}^{r'})\subset \Lambda_{L,S,T}^r$$
and 
$$\Phi_{V',V}(\varep_{L,S',V'})=\varep_{L,S,V}.$$
\end{proposition}

\begin{proof}
Put $\Phi=\Phi_{V',V}$, for simplicity. First, we prove that 
\begin{eqnarray}
\Phi(\Lambda_{L,S',T}^{r'})\otimes_\bbZ\bbQ =\Lambda_{L,S,T}^r\otimes_\bbZ\bbQ. \label{phiq}
\end{eqnarray}
There is a split exact sequence of $\bbQ[\cG_L]$-modules:
$$0\longrightarrow \cO_{L,S,T}^\times \otimes_\bbZ\bbQ \longrightarrow \cO_{L,S',T}^\times \otimes_\bbZ\bbQ \stackrel{\bigoplus_{v \in S'\setminus S}\widetilde w}{\longrightarrow} \bigoplus_{v\in S'\setminus S}\bbQ[\cG_L] \longrightarrow 0,$$
where $\widetilde w =\sum_{\sigma \in \cG_L}\ord_w(\sigma(\cdot))\sigma^{-1}$. So we can choose a submodule $M \subset \cO_{L,S',T}^\times \otimes_\bbZ\bbQ$ such that 
$$\cO_{L,S',T}^\times \otimes_\bbZ\bbQ =(\cO_{L,S,T}^\times \otimes_\bbZ\bbQ)\oplus M$$
and 
$$\bigoplus_{v\in S'\setminus S}\widetilde w: M{\longrightarrow} \bigoplus_{v\in S'\setminus S}\bbQ[\cG_L]$$
is an isomorphism. Therefore, we have 
$$(\bigwedge_{\cG_L}^{r'}\cO_{L,S',T}^\times)\otimes_\bbZ\bbQ =\bigoplus_{i=0}^{r'}((\bigwedge_{\cG_L}^i\cO_{L,S,T}^\times)\otimes_\bbZ\bbQ)\otimes_{\bbQ[\cG_L]}\bigwedge_{\bbQ[\cG_L]}^{r'-i}M.$$
If $i > r$ then $\Phi(((\bigwedge_{\cG_L}^i\cO_{L,S,T}^\times)\otimes_\bbZ\bbQ)\otimes_{\bbQ[\cG_L]}\bigwedge_{\bbQ[\cG_L]}^{r'-i}M)=0$, and if $i<r$ then $\bigwedge_{\bbQ[\cG_L]}^{r'-i}M=0$. Hence we have 
$$\Phi(\bigwedge_{\cG_L}^{r'}\cO_{L,S',T}^\times)\otimes_\bbZ\bbQ=(\bigwedge_{\cG_L}^r\cO_{L,S,T}^\times) \otimes_\bbZ\bbQ.$$
Now (\ref{phiq}) follows by noting that $r(\chi,S')=r(\chi,S)+r'-r$ for every $\chi \in \widehat \cG_L$.

For the first assertion, by (\ref{phiq}), it is enough to prove that 
$\Phi(\bigcap_{\cG_L}^{r'}\cO_{L,S',T}^\times) \subset \bigcap_{\cG_L}^r\cO_{L,S,T}^\times$. Since $\cO_{L,S',T}^\times/\cO_{L,S,T}^\times$ 
is torsion-free, we have a surjection $\Hom_{\cG_L}(\cO_{L,S',T}^\times,\bbZ[\cG_L]) \rightarrow \Hom_{\cG_L}(\cO_{L,S,T}^\times,\bbZ[\cG_L])$. 
Now the assertion follows from the definition of Rubin's lattice.

For the second assertion, it is enough to show that 
$$R_V(\Phi(\varep_{L,S',V'}))=\Theta_{L,S,T}^{(r)}(0).$$
It is easy to see that for $v\in V'\setminus V$
$$\log\N v \sum_{\sigma \in \cG_L}\ord_w(\sigma(\cdot))\sigma^{-1}=v^\ast \circ \lambda_{L,S'},$$
and also that 
$$\Theta_{L,S',T}^{(r')}(0)=(\prod_{v\in V'\setminus V}\log\N v) \Theta_{L,S,T}^{(r)}(0).$$
Therefore, we have 
\begin{eqnarray}
R_V(\Phi(\varep_{L,S',V'}))&=& (\prod_{v\in V'\setminus V}\log\N v)^{-1}R_{V'}(\varep_{L,S',V'}) \nonumber \\
&=&(\prod_{v\in V'\setminus V}\log\N v)^{-1}\Theta_{L,S',T}^{(r')}(0) \nonumber \\
&=&\Theta_{L,S,T}^{(r)}(0). \nonumber
\end{eqnarray}
\end{proof}

\subsection{Refined conjectures} \label{refconj}
In this subsection, we propose the main conjectures. We keep the notations in \S \ref{not}. We also keep on assuming Conjecture \ref{rsconj} is true 
for all $(L,S,V) \in \Omega$.
Fix $(L,S,V),(L',S',V') \in \Omega$ such that $L \subset L'$, $S \subset S'$, and $V \supset V'$. We also use the notations defined in \S \ref{alg}, taking $G=\cG_{L'}$ and $H=\Gal(L'/L)$. For convenience, we record the list of the notations here (some new notations are added).
\begin{itemize}
\item{$\cG_L=\Gal(L/k)$,}
\item{$\cG_{L'}=\Gal(L'/k)$,}
\item{$G(L'/L)=\Gal(L'/L)$,}
\item{$r=|V|$,}
\item{$r'=|V'|$,}
\item{$\varep_{L,S,V} \in \bigcap_{\cG_L}^r \cO_{L,S,T}^\times \ (\mbox{resp. }\varep_{L',S',V'}\in \bigcap_{\cG_{L'}}^{r'}\cO_{L',S',T}^\times)$: Rubin-Stark unit for $(L,S,V)$ (resp. $(L',S',V')$) (see \S \ref{secrsconj}).}
\item{$d=r-r'(\geq 0)$,}
\item{$I_{L'/L}=I_{G(L'/L)}=\Ker(\bbZ[\cG_{L'}]\longrightarrow \bbZ[\cG_L])$,}
\item{$I(L'/L)=I(G(L'/L))=\Ker(\bbZ[G(L'/L)]\longrightarrow \bbZ)$.}
\end{itemize}
For $n\in\bbZ_{\geq0},$
\begin{itemize}
\item{$Q_{L'/L}^n=Q_{G(L'/L)}^n=I_{L'/L}^n/I_{L'/L}^{n+1}$,}
\item{$Q(L'/L)^n=Q(G(L'/L))^n= I(L'/L)^n/I(L'/L)^{n+1}.$}
\end{itemize}
Recall that there is a natural isomorphism
$$\bbZ[\cG_L]\otimes_\bbZ Q(L'/L)^n \simeq Q_{L'/L}^n$$
(see (\ref{eqaug})).

Recall the definition of ``higher norm'' (Definition \ref{defnorm}). In the case $r' \geq 1$, the $d$-th norm  
$$\N_{L'/L}^{(r',d)}=\N_{G(L'/L)}^{(r',d)} : \bigcap_{\cG_{L'}}^{r'}\cO_{L',S',T}^\times \longrightarrow (\bigcap_{\cG_{L'}}^{r'}\cO_{L',S',T}^\times )\otimes_\bbZ \bbZ[G(L'/L)]/I(L'/L)^{d+1}$$
is defined by 
$$\N_{L'/L}^{(r',d)}(a)=\sum_{\sigma \in G(L'/L)}\sigma a \otimes \sigma^{-1},$$
and in the case $r'=0$, $\N_{L'/L}^{(0,d)}$ is defined to be the natural map 
$$\bbZ[\cG_{L'}] \longrightarrow \bbZ[\cG_{L'}]/I_{L'/L}^{d+1}.$$

In the case $r'\geq 1$, define 
$$i : (\bigcap_{\cG_L}^{r'}\cO_{L,S',T}^\times) \otimes_\bbZ Q(L'/L)^d \hookrightarrow (\bigcap_{\cG_{L'}}^{r'}\cO_{L',S',T}^\times) \otimes_\bbZ \bbZ[G(L'/L)]/I(L'/L)^{d+1}$$
to be the canonical injection in Lemma \ref{inj}. In the case $r'=0$, define 
$$i: (\bigcap_{\cG_L}^0 \cO_{L,S',T}^\times) \otimes_\bbZ Q(L'/L)^d \simeq Q_{L'/L}^d \hookrightarrow \bbZ[\cG_{L'}]/I_{L'/L}^{d+1}$$
to be the inclusion.

\begin{conjecture} \label{descent}
$$\N_{L'/L}^{(r',d)}(\varep_{L',S',V'}) \in \Im i.$$

\end{conjecture}

\begin{remark}
When $d=0$, Conjecture \ref{descent} is true by Remarks \ref{reminj} and \ref{remnorm}.
\end{remark}

\begin{remark}
Conjecture \ref{descent} is related to the Kolyvagin's derivative construction, which is important in the theory of Euler systems (\cite{K}, \cite{R}) and Mazur-Rubin's Kolyvagin systems (\cite{MR1}). See Remark \ref{remkoly} for the detail.
\end{remark}

For $v\in V$, define
$$\varphi_v=\varphi_{v,L'/L} : L^\times \longrightarrow Q_{L'/L}^1$$
by $\varphi_v(a)=\sum_{\sigma\in\cG_L}(\rec_w(\sigma a)-1)\sigma^{-1},$ where $\rec_w$ is the local reciprocity map at $w$ (recall that $w$ is the fixed place of $L$ lying above $v$, see \S \ref{not}). Note that, by Proposition \ref{propphi}, $\bigwedge_{v \in V \setminus V'}\varphi_v \in \bigwedge_{\cG_L}^d \Hom_{\cG_L}(\cO_{L,S,T}^\times, Q_{L'/L}^1)$ induces a morphism
$$\bigcap_{\cG_L}^r\cO_{L,S,T}^\times \longrightarrow( \bigcap_{\cG_L}^{r'}\cO_{L,S,T}^\times )\otimes_\bbZ Q(L'/L)^d.$$

We define $\sgn(V,V')=\pm 1$ by 
$$(\bigwedge_{v\in V'}v^\ast)\circ(\bigwedge_{v\in V\setminus V'}v^\ast)=\sgn(V,V')\bigwedge_{v\in V}v^\ast \quad \mbox{in}\quad\Hom_{\cG_L}(\bigwedge_{\cG_L}^{r}Y_{L,S}, \bbZ[\cG_L]).$$

The following conjecture predicts that $\N_{L'/L}^{(r',d)}(\varep_{L',S',V'})$ is described in terms of $\varep_{L,S,V}$.

\begin{conjecture} \label{hnr}
Conjecture \ref{descent} holds, and we have  
$$i^{-1}(\N_{L'/L}^{(r',d)}(\varep_{L',S',V'}))=\sgn(V,V')(\prod_{v\in S'\setminus S}(1-\Fr_v^{-1}))(\bigwedge_{v \in V\setminus V'}\varphi_v)(\varep_{L,S,V}).$$
\end{conjecture}

\begin{remark} \label{remhnr}
When $d=0$, Conjecture \ref{hnr} is true by ``norm relation'' (Proposition \ref{nr}). (See Remarks \ref{reminj} and \ref{remnorm}.)
\end{remark}

\begin{remark} \label{gross}
When $r'=0$, by Remark \ref{rzero}, one sees that Conjecture \ref{hnr} is equivalent to the ``Gross-type refinement of the Rubin-Stark conjecture" (\cite[Conjecture 5.3.3]{P}), which 
generalizes Gross's conjecture (\cite[Conjecture 4.1]{G}), see \cite[Proposition 5.3.6]{P}.
\end{remark}

\begin{remark} \label{hnrdar}
When $r'=1$, Conjecture \ref{hnr} is closely related to Darmon's conjecture (\cite[Conjecture 4.3]{D}). The detailed explanation is 
given in \S \ref{appl}.
\end{remark}

\begin{proposition} \label{propred}
It is sufficient to prove Conjecture \ref{hnr} in the following case: \\ 
\quad $S=S'$, \\
\quad $r= \min \{ |S|-1, |\{ v\in S \ | \ v \mbox{ splits completely in }L \}| \} =:r_{L,S}$, \\
\quad $r'=\min\{ |S|-1, |\{ v\in S \ | \ v \mbox{ splits completely in } L' \}| \} =:r_{L',S}$.
\end{proposition}

\begin{proof}
From Proposition \ref{nr}, we may assume $S=S'$. When $r < r_{L,S}$ and $r' < r_{L',S}$, Conjecture \ref{hnr} is trivially true (see Remark \ref{rstriv}). When $r<r_{L,S}$ and $r'=r_{L',S}$, we have 
$$\N_{L'/L}^{(r',d)}(\varep_{L',S,V'})=0$$
if Conjecture \ref{hnr} is true when $r=r_{L,S}$ and $r'=r_{L',S}$. When $r=r_{L,S}$ and $r'<r_{L',S}$, we prove 
$$(\bigwedge_{v \in V\setminus V'}\varphi_v)(\varep_{L,S,V})=0.$$
If there exists $v \in V\setminus V'$ which splits completely in $L'$, this is clear. If all $v \in V \setminus V'$ don't split completely in $L'$, then there exists $v' \in S\setminus V$ which splits completely in $L'$, and we must have $V=S\setminus \{v'\}$.
By the product formula, we see that 
$$\sum_{v \in S\setminus V'}\varphi_{v,L'/L}=0 \quad \text{on} \quad \cO_{k,S,T}^\times.$$
Note that $\varep_{L,S,V} \in e_1(\bbQ\otimes_\bbZ \bigwedge_{\cG_L}^r \cO_{L,S,T}^\times)$ in this case. Hence, choosing any $v'' \in V \setminus V'$, we have 
$$(\bigwedge_{v\in V\setminus V'}\varphi_v)(\varep_{L,S,V})=\pm (\bigwedge_{v\in (S\setminus \{ v''\})\setminus V'}\varphi_{v})(\varep_{L,S,V}),$$
and the right hand side is $0$ since $v'$ splits completely in $L'$. 
\end{proof}

From now on we assume $S=S'$, $r=r_{L,S}$, and $r'=r_{L',S}$. 

\begin{proposition} \label{unram}
If every place in $V\setminus V'$ is finite and unramified in $L'$, then Conjecture \ref{hnr} is true.
\end{proposition}

\begin{proof}
We treat the case $r'\geq 1$. The proof for $r'=0$ is similar.

Put $W:=V\setminus V'$ for simplicity. Note that $(L',S\setminus W, V' ) \in \Omega$. By Proposition \ref{nr}, we have 
$$\varep_{L',S,V'}=\prod_{v\in W}(1-\Fr_v^{-1})\varep_{L',S\setminus W,V'}.$$
Hence, we have 
\begin{eqnarray}
\N_{L'/L}^{(r',d)}(\varep_{L',S,V'})&=& \sum_{\sigma \in G(L'/L)}\sigma \prod_{v\in W}(1-\Fr_v^{-1})\varep_{L',S\setminus W,V'} \otimes \sigma^{-1} \nonumber \\
&=& \sum_{\sigma \in G(L'/L)}\sigma \varep_{L',S\setminus W,V'} \otimes \sigma^{-1} \prod_{v\in W}(1-\Fr_v^{-1})\nonumber \\
&=&\N_{L'/L}\varep_{L',S\setminus W,V'} \prod_{v\in W}(\Fr_v-1) \nonumber \\
&\in& (\N_{L'/L}\bigcap_{\cG_{L'}}^{r'}\cO_{L',S,T}^\times )\otimes_{\bbZ}Q(L'/L)^d. \nonumber 
\end{eqnarray}
For every $v\in W$, we have 
$$\varphi_{v}=\sum_{\sigma \in \cG_L}\ord_{w}(\sigma (\cdot))\sigma^{-1}(\Fr_v-1).$$
(See \cite[Proposition 13, Chpt. XIII]{Se}.) So, by Proposition \ref{ordr}, we have 
$$\sgn(V,V')(\bigwedge_{v\in W}\varphi_v)(\varep_{L,S,V})=\varep_{L,S\setminus W,V'}\prod_{v\in W}(\Fr_v-1).$$
By Proposition \ref{nr} and Remark \ref{reminj}, we have 
$$\N_{L'/L}\varep_{L',S\setminus W,V'}\prod_{v\in W}(\Fr_v-1)=i(\varep_{L,S\setminus W,V'}\prod_{v\in W}(\Fr_v-1)),$$
hence the proposition follows. 
\end{proof}

The formulation of the following conjecture is a slight modification of \cite[Theorem 3.1]{B} (see also Theorem \ref{bthm} and Remark \ref{brem}).

\begin{conjecture} \label{bconj}
For every $\Phi \in \bigwedge_{\cG_{L'}}^{r'} \Hom_{\cG_{L'}}(\cO_{L',S,T}^\times, \bbZ[\cG_{L'}])$, we have
$$\Phi(\varep_{L',S,V'})\in I_{L'/L}^d$$
and 
$$\Phi(\varep_{L',S,V'})=\sgn(V,V')\Phi^{G(L'/L)}((\bigwedge_{v \in V\setminus V'}\varphi_v)(\varep_{L,S,V})) \quad \mbox{in}\quad Q_{L'/L}^d.$$
\end{conjecture}

The following conjecture is motivated by the property of the higher norm described in Proposition \ref{propnorm}.

\begin{conjecture} \label{phiconj}
If Conjecture \ref{descent} holds, then we have 
$$\Phi(\varep_{L',S,V'})=\Phi^{G(L'/L)}(i^{-1}(\N_{L'/L}^{(r',d)}(\varep_{L',S,V'}))) \quad \mbox{in} \quad Q_{L'/L}^d$$
for every $\Phi \in \bigwedge_{\cG_{L'}}^{r'}\Hom_{\cG_{L'}}(\cO_{L',S,T}^\times,\bbZ[\cG_{L'}])$. 
\end{conjecture}

\begin{remark} \label{remphiconj}
When $d=0$ or $r'=0$ or $1$, Conjecture  \ref{phiconj} is true by Proposition \ref{propnorm}.
\end{remark}

\subsection{Relation among the conjectures}
We keep on assuming $S=S'$, $r=r_{L,S}$, and $r'=r_{L',S}$.

\begin{theorem} \label{rel}
Assume Conjecture \ref{phiconj} holds. Then, Conjecture \ref{hnr} holds if and only if  Conjectures \ref{descent} and \ref{bconj} hold. 
\end{theorem}

\begin{proof}
The ``only if" part is clear. We prove the ``if" part. Suppose that Conjectures \ref{descent} and \ref{bconj} hold. Then, for every $\Phi \in \bigwedge_{\cG_{L'}}^{r'}\Hom_{\cG_{L'}}(\cO_{L',S,T}^\times,\bbZ[\cG_{L'}])$, we have 
$$\Phi^{G(L'/L)}(i^{-1}(\N_{L'/L}^{(r',d)}(\varep_{L',S,V'})))= \sgn(V,V')\Phi^{G(L'/L)}((\bigwedge_{v \in V\setminus V'}\varphi_v)(\varep_{L,S,V}))\quad \text{in} \quad Q_{L'/L}^d$$
by Conjectures \ref{bconj} and \ref{phiconj}. By Theorem \ref{thminj}, the map
$$(\bigcap_{\cG_{L}}^{r'}\cO_{L,S,T}^\times) \otimes_\bbZ Q(L'/L)^d \longrightarrow \Hom_{\cG_{L'}}(\bigwedge_{\cG_{L'}}^{r'}\Hom_{\cG_{L'}}(\cO_{L',S,T}^\times,\bbZ[\cG_{L'}]),Q_{L'/L}^d)$$
defined by $\alpha \mapsto (\Phi \mapsto \Phi^{G(L'/L)}(\alpha))$ is injective. Hence we have
$$i^{-1}(\N_{L'/L}^{(r',d)}(\varep_{L',S,V'}))=\sgn(V,V')(\bigwedge_{v \in V\setminus V'}\varphi_v)(\varep_{L,S,V}).$$

\end{proof}

\begin{remark}
Since Conjecture \ref{hnr} is closely related to Darmon's conjecture, as we mentioned in Remark \ref{hnrdar}, Theorem \ref{rel} gives a relation between Darmon's conjecture and Burns's conjecture (Conjecture \ref{bconj}). In \cite[Theorem 6.14]{H}, Hayward established a connection between these conjectures: he proved that Darmon's conjecture gives a ``base change statement" for Burns's conjecture. 
More precisely, consider a real quadratic field $L$ and a real abelian field $\widetilde L$ which is disjoint to $L$. Put $L':=L \widetilde L$. Then Hayward proved that, assuming Darmon's conjecture for $L$, Burns's conjecture for $\widetilde L/\bbQ$ implies Burns's conjecture for $L'/L$ up to a power of $2$. 
On the other hand, Theorem \ref{rel} gives an equivalence of Burns's conjecture and Darmon's conjecture, assuming Conjectures \ref{descent} and \ref{phiconj}. 
\end{remark}

\begin{remark} \label{remp}
One can formulate for any prime number $p$ the ``$p$-part" of Conjectures \ref{descent}, \ref{hnr}, \ref{bconj}, and \ref{phiconj} in the obvious way. One sees that 
the ``$p$-part" of Theorem \ref{rel} is also valid, namely, assuming the ``$p$-part" of Conjecture \ref{phiconj}, the ``$p$-part" of Conjecture \ref{hnr} holds if and only if the ``$p$-part" of Conjectures \ref{descent} and \ref{bconj} hold. 
\end{remark}

The following theorem gives evidence for the validity of Conjecture \ref{bconj}.

\begin{theorem}[Burns {\cite[Theorem 3.1]{B}}] \label{bthm}
If the conjecture in \cite[\S 6.3]{B} holds for $L'/k$, then we have 
$$\Phi(\varep_{L',S,V'})\in I_{L'/L}^d$$
for every $\Phi \in \bigwedge_{\cG_{L'}}^{r'} \Hom_{\cG_{L'}}(\cO_{L',S,T}^\times, \bbZ[\cG_{L'}])$ 
and an equality
$$\Phi(\varep_{L',S,V'})=\sgn(V,V')\Phi^{G(L'/L)}((\bigwedge_{v \in V\setminus V'}\varphi_v)(\varep_{L,S,V}))$$
in $\Coker (\bigwedge_{v\in V \setminus V'}\varphi_v : (\bigwedge_{\cG_L}^d L_T^\times)_{\tors} \rightarrow Q_{L'/L}^d )$, where $L_T^\times$ is the 
subgroup of $L^\times$ defined by 
$$L_T^\times=\{ a\in L^\times \ | \ \ord_w(a-1)>0 \mbox{ for all }w\in T_L \}.$$ 
\end{theorem}

\begin{remark} \label{remetnc}
In the number field case, as Burns mentioned in \cite[Remark 6.2]{B}, the conjecture in \cite[\S 6.3]{B} for $L'/k$ is equivalent to the ``equivariant Tamagawa number conjecture 
(ETNC)" (\cite[Conjecture 4 (iv)]{BF}) for the pair $(h^0(\Spec(L')),\bbZ[\cG_{L'}] )$, and known to be true if $L'$ is an abelian extension over $\bbQ$ by the works of Burns, Greither, and Flach (\cite{BG}, \cite{F}). 
\end{remark}

\begin{remark} \label{brem}
In \cite[Theorem 3.1]{B}, Burns actually proved more: let
$$I_{L'/L}^S=
\begin{cases}
\prod_{v\in V\setminus V'}I_v &\text{if $d>0$,} \\
\bbZ[\cG_{L'}] &\text{if $d=0$,} 
\end{cases} 
$$
where $I_v=\Ker(\bbZ[\cG_{L'}] \rightarrow \bbZ[\cG_{L'}/\cG_v])$ and $\cG_v$ is the decomposition group of $w$ in $G(L'/L)$. Then Burns proved  that, under the assumption that the conjecture in \cite[\S 6.3]{B} holds for $L'/k$, $\Phi(\varep_{L',S,V'})\in I_{L'/L}^S$ for every $\Phi \in \bigwedge_{\cG_{L'}}^{r'} \Hom_{\cG_{L'}}(\cO_{L',S,T}^\times, \bbZ[\cG_{L'}])$ and an equality
$$\Phi(\varep_{L',S,V'})=\sgn(V,V')\Phi^{G(L'/L)}((\bigwedge_{v \in V\setminus V'}\varphi_v)(\varep_{L,S,V}))$$
holds in $\Coker (\bigwedge_{v\in V \setminus V'}\varphi_v : (\bigwedge_{\cG_L}^d L_T^\times)_{\tors} \rightarrow I_{L'/L}^S/I_{L'/L}I_{L'/L}^S )$.
\end{remark}

\begin{proposition} \label{prop}
$$(\bigwedge_{\cG_L}^dL_T^\times)_{\tors} \otimes_\bbZ \bbZ[\frac1{|\cG_L|}]=0.$$
\end{proposition}
\begin{proof}
Note that 
$$\bigwedge_{\cG_L}^dL_T^\times=\dlim \bigwedge_{\cG_L}^d \cO_{L,\Sigma,T}^\times,$$
where $\Sigma$ runs over all finite sets of places of $k$, which contains all the infinite places and places ramifying in $L$, and is disjoint from $T$, and the direct limit is taken by the map induced by the inclusion $\cO_{L,\Sigma,T} \hookrightarrow \cO_{L,\Sigma',T}$ ($\Sigma \subset \Sigma'$). So it is sufficient to prove that for such $\Sigma$, $\bigwedge_{\cG_L}^d \cO_{L,\Sigma,T}^\times \otimes_\bbZ \bbZ[\frac1{|\cG_L|}]$ is torsion-free. 
Since $\cO_{L,S,T}^\times$ is torsion-free, we see that $\cO_{L,\Sigma,T}^\times$ is also torsion-free. It is well-known that a finitely generated $\bbZ[\frac1{|\cG_L|}][\cG_L]$-module is locally free if and only if it is torsion-free. So we see that $\cO_{L,\Sigma,T}^\times \otimes_\bbZ \bbZ[\frac1{|\cG_L|}]$ is locally free $\bbZ[\frac1{|\cG_L|}][\cG_L]$-module. Hence $\bigwedge_{\cG_L}^d \cO_{L,\Sigma,T}^\times \otimes_\bbZ \bbZ[\frac1{|\cG_L|}]$ is also locally free, so it is torsion-free.
\end{proof}

Combining Theorem \ref{rel}, Theorem \ref{bthm}, and Proposition \ref{prop}, we have the following theorem (see also Remark \ref{remp}).

\begin{theorem} \label{mainthm}
Let $p$ be a prime number not dividing $|\cG_L|$. Assume the ``$p$-part" of Conjecture \ref{phiconj} holds. If the conjecture in \cite[\S 6.3]{B} for $L'/k$ and the ``$p$-part" of Conjecture \ref{descent} hold, then the ``$p$-part" of Conjecture \ref{hnr} holds.
\end{theorem}

\section{An application} \label{appl}
In this section, as an application of Theorem \ref{mainthm}, we give another proof of the ``except $2$-part" of Darmon's conjecture (Mazur-Rubin's theorem, see Theorem \ref{mrthm}). 
\subsection{Darmon's conjecture}
We review the slightly modified version of Darmon's conjecture, formulated in \cite{MR2}. First, we fix the following:
\begin{itemize}
\item{a bijection $\{ \mbox{all the places of }\bbQ\} \simeq \bbZ_{\geq0}$ such that $\infty$ (the infinite place of $\bbQ$) corresponds to $0$  (from this, we endow a total order on $\{ \mbox{all the places of }\bbQ\}$),}
\item{for each place $v$ of $\bbQ$, a place of $\overline \bbQ$ lying above $v$.}
\end{itemize}
Let $F/\bbQ$ be a real quadratic field, and $\chi$ be the corresponding Dirichlet character with conductor $f$. 
Let $n$ be a square-free product of primes not dividing $f$. Put 
$$n_{\pm}=\prod_{\ell | n,\chi(\ell)=\pm 1}\ell,$$
(throughout this section, $\ell$ always denotes a prime number), and let $\nu_\pm$ be the number of prime divisors of $n_\pm$. 
Let 
$$\alpha_n=\left( \sum_{\sigma \in \Gal(\bbQ(\mu_{nf})/\bbQ(\mu_n))}\chi(\sigma)\sigma \right) (1-\zeta_{nf}) \in F(\mu_n)^\times,$$
where for any positive integer $m$, $\mu_m$ denotes the group of $m$-th roots of unity in $\overline \bbQ$, and $\zeta_m =e^{\frac{2\pi i}{m}}$ (the embedding $\overline \bbQ \hookrightarrow \bbC$ is fixed above). Put 
$$\theta_n=\sum_{\sigma \in \Gal(F(\mu_n)/F)}\sigma \alpha_n \otimes \sigma \in F(\mu_n)^\times \otimes_\bbZ \bbZ[\Gal(F(\mu_n)/F)].$$
Let $I_n$ be the augmentation ideal of $\bbZ[\Gal(F(\mu_n)/F)]$. Note that the natural map 
$$F^\times \otimes_\bbZ I_n^{\nu_+}/I_n^{\nu_+ +1} \otimes_\bbZ \bbZ[\frac12] \longrightarrow F(\mu_n)^\times \otimes_\bbZ I_n^{\nu_+}/I_n^{\nu_+ +1} \otimes_\bbZ \bbZ[\frac12]$$
is injective (see \cite[Lemma 9.2]{D}).

\begin{proposition}[Darmon {\cite[Theorem 4.5 (2)]{D}}] \label{darprop}
We have $\theta_n \in F(\mu_n)^\times \otimes_\bbZ I_n^{\nu_+}$ and the image of $\theta_n$ in $F(\mu_n)^\times \otimes_\bbZ I_n^{\nu_+}/I_n^{\nu_+ +1} \otimes_\bbZ \bbZ[\frac12]$ belongs to $F^\times \otimes_\bbZ I_n^{\nu_+}/I_n^{\nu_+ +1} \otimes_\bbZ \bbZ[\frac12]$.
\end{proposition}

We often denote the image of $\theta_n$ in $F(\mu_n)^\times \otimes_\bbZ I_n^{\nu_+}/I_n^{\nu_+ +1} \otimes_\bbZ \bbZ[\frac12]$ also by $\theta_n$. 

Next, write $n_+=\prod_{i=1}^{\nu_+}\ell_i$ so that $\ell_1 \prec \cdots \prec \ell_{\nu_+}$ (``$\prec$" is the total order fixed above), and let $\lambda_i$ be the fixed place of $F$ lying above $\ell_i$. Let $\lambda_0$ 
be the fixed place of $F$ lying above $\infty$. Let $\tau$ be the generator of $\Gal(F/\bbQ)$. 
Take $u_0,\ldots,u_{\nu_+} \in \cO_F[\frac1n]^\times$ such that $\{(1-\tau)u_i\}_{0\leq i \leq \nu_+}$ forms a $\bbZ$-basis of $(1-\tau)\cO_F[\frac1n]^\times$ (which is in fact a free abelian group of rank $\nu_+ +1$, see \cite[Lemma 3.2 (ii)]{MR2}), and $\det(\log|(1-\tau)u_i|_{\lambda_j})_{0\leq i,j \leq \nu_+} >0$. Put
$$R_n=(-1)^{\nu_+}(\varphi_{\ell_1}^1 \wedge \cdots \wedge \varphi_{\ell_{\nu_+}}^1)((1-\tau)u_0 \wedge \cdots \wedge (1-\tau)u_{\nu_+}) \in (1-\tau)\cO_F[\frac1n]^\times \otimes_\bbZ I_n^{\nu_+}/I_n^{\nu_+ +1},$$
where 
$$\varphi_{\ell_i}^1 : F^\times \longrightarrow I_n/I_n^2$$
is defined by $\varphi_{\ell_i}^1=\rec_{\lambda_i}(\cdot)-1$, where $\rec_{\lambda_i} : F^\times \rightarrow \Gal(F(\mu_n)/F)$ is the local reciprocity map at $\lambda_i$. Note that we have 
$$R_n=\det \left(
\begin{array}{ccc}
(1-\tau)u_0 & \cdots & (1-\tau)u_{\nu_+} \\
\varphi_{\ell_1}^1((1-\tau)u_0) & \cdots &\varphi_{\ell_1}^1((1-\tau)u_{\nu_+}) \\
\vdots & \ddots &\vdots \\
\varphi_{\ell_{\nu_+}}^1((1-\tau)u_0) & \cdots &\varphi_{\ell_{\nu_+}}^1((1-\tau)u_{\nu_+})
\end{array}
\right).$$
Finally, let $h_n$ denote the $n$-class number of $F$, i.e. the order of the Picard group of $\Spec \cO_F[\frac1n]$.

Now Darmon's conjecture is stated as follows.

\begin{conjecture}[Darmon {\cite[Conjecture 4.3]{D}}, {\cite[Conjecture 3.8]{MR2}}] \label{dconj}
$$\theta_n=-2^{\nu_-}h_nR_n \quad \mbox{in}\quad (F(\mu_n)^\times/\{ \pm1 \})\otimes_\bbZ I_n^{\nu_+}/I_n^{\nu_+ +1}.$$
\end{conjecture}

Mazur and Rubin proved that this conjecture holds ``except 2-part".

\begin{theorem}[Mazur-Rubin {\cite[Theorem 3.9]{MR2}}] \label{mrthm}
$$\theta_n=-2^{\nu_-}h_nR_n \quad \mbox{in} \quad F^\times\otimes_\bbZ I_n^{\nu_+}/I_n^{\nu_+ +1}\otimes_\bbZ \bbZ[\frac12].$$
\end{theorem}

\subsection{Proof of Theorem \ref{mrthm}}
We keep the notations in the previous subsection, and also use the notations defined in \S \ref{secconj}. We specialize the general setting of \S \ref{secconj} into the following: 
\begin{itemize}
\item{$k=\bbQ$,}
\item{$L=F$ (a real quadratic field),}
\item{$L'=F(\mu_n)^+$ (the maximal real subfield of $F(\mu_n)$),}
\item{$S=S'=\{ \infty \} \cup \{ \mbox{primes dividing }nf \}$,}
\item{$V=\{ \infty \}\cup \{\mbox{primes dividing }n_+ \}$,}
\item{$V'=\{ \infty \}$,}
\item{$T$: a finite set of places of $\bbQ$ such that 
\begin{itemize}
\item{$S\cap T = \emptyset$,}
\item{$\cO_{L',S,T}^\times$ is torsion-free.}
\end{itemize}
}
\end{itemize}

Then one sees that $(L,S,V), (L',S,V') \in \Omega=\Omega(\bbQ,T)$. 

It is known that the Rubin-Stark conjecture (Conjecture \ref{rsconj}) for all the triples in $\Omega$ holds (\cite[Theorem A]{B}). Let 
$$\varep_T=\varep_{L,S,T,V}\in \bigcap_{\cG_L}^{\nu_+ +1}\cO_{L,S,T}^\times \quad(\mbox{resp. }\varep_T'=\varep_{L',S,T,V'} \in \bigcap_{\cG_{L'}}^1 \cO_{L',S,T}^\times =\cO_{L',S,T}^\times)$$
denote the Rubin-Stark unit for the triple $(L,S,V)$ (resp. $(L',S,V')$) (later we will vary $T$, so we keep in the notation the dependence on $T$).

Note that, since $r'=1$ in this setting, Conjecture \ref{phiconj} holds (see Remark \ref{remphiconj}). Note also that, since $F(\mu_n)^+$ is abelian over $\bbQ$, the conjecture in \cite[\S 6.3]{B} holds (see Remark \ref{remetnc}). So, by Theorem \ref{mainthm}, if we show the ``except $2$-part" of Conjecture \ref{descent}, then we know that the ``except $2$-part" of Conjecture \ref{hnr} holds. The ``except $2$-part" of Conjecture \ref{hnr} implies Theorem \ref{mrthm}, as we will explain below. Unfortunately, we cannot prove Conjecture \ref{descent} completely. Instead, we prove the following weak version of it:

\begin{proposition} \label{propdes}
Let $\Sigma$ be a finite set of places of $\bbQ$, which contains $S$ and is disjoint from $T$. If $\Sigma$ is large enough, then we have 
$$\N_{L'/L}^{(1, \nu_+)}(\varep_T') \in \cO_{L,\Sigma,T}^\times \otimes_\bbZ Q(L'/L)^{\nu_+}\otimes _\bbZ \bbZ[\frac12].$$
\end{proposition}

The proof of this proposition is given in \S \ref{proof2}. This proposition gives sufficient ingredients to prove the ``except $2$-part" of Conjecture \ref{hnr}: using Proposition \ref{propnorm}, Theorem \ref{thminj}, Theorem \ref{bthm}, and Proposition \ref{prop}, we have the following

\begin{theorem} \label{hnr2}
$$\N_{L'/L}^{(1,\nu_+)}(\varep_T')=(-1)^{\nu_+}(\bigwedge_{\ell | n_+}\varphi_\ell)(\varep_T)\quad \mbox{in} \quad L^\times \otimes_\bbZ Q(L'/L)^{\nu_+}\otimes_\bbZ \bbZ[\frac12].$$
\end{theorem} 

We will deduce Theorem \ref{mrthm} from Theorem \ref{hnr2} by varying the set $T$. 

The following proposition is well-known.

\begin{proposition} \label{propT}
There exists a finite family $\cT$ of $T$ such that $S\cap T =\emptyset$ and $\cO_{L',S,T}^\times$ is torsion-free, and for every $T\in \cT$, 
there is an $a_T \in \bbZ[\cG_{L'}]$ such that 
$$2=\sum_{T \in \cT}a_T \delta_T \quad \mbox{in} \quad \bbZ[\cG_{L'}],$$
where $\delta_T=\prod_{\ell \in T}(1-\ell \Fr_\ell^{-1}) \in \bbZ[\cG_{L'}]$.
\end{proposition}
For the proof, see \cite[Lemme 1.1, Chpt. IV]{T}. Take such a family $\cT$ and $a_T$ for each $T\in\cT$. The following lemma will be proved in \S \ref{proof2}.

\begin{lemma} \label{Tlem}
{\rm(i)} 
$$(1-\tau)\sum_{T\in\cT}a_T\varep_T'=\N_{L(\mu_n)/L'}(\alpha_n) \quad \mbox{in} \quad L'^\times/\{ \pm 1 \},$$
where $\tau$ is regarded as the generator of $\Gal(L'/\bbQ(\mu_n)^+)$. \\  
{\rm(ii)}
$$(1-\tau)\sum_{T\in\cT}a_T\varep_T=(-1)^{\nu_++1}2^{\nu_-}h_n(1-\tau)u_0 \wedge \cdots \wedge u_{\nu_+} \quad \mbox{in} \quad\bbQ \otimes_\bbZ\bigwedge_{\cG_L}^{\nu_++1}\cO_{L,S}^\times.$$
\end{lemma}

The following lemma is easily verified, so we omit the proof.

\begin{lemma} \label{corrlem}
The natural map $\Gal(L(\mu_n)/L) \rightarrow G(L'/L)$ induces an isomorphism 
$$\pi : L^\times \otimes_\bbZ I_n^{\nu_+}/I_n^{\nu_++1}\otimes_\bbZ \bbZ[\frac12] \stackrel{\sim}{\longrightarrow} L^\times \otimes_\bbZ Q(L'/L)^{\nu_+}\otimes_\bbZ \bbZ[\frac12],$$
and we have 
$$\pi(\theta_n)=(-1)^{\nu_+}\sum_{\sigma \in G(L'/L)}\sigma \N_{L(\mu_n)/L'}(\alpha_n) \otimes \sigma^{-1},$$
and 
$$\pi(-2^{\nu_-}h_nR_n)=(-1)^{\nu_++1}2^{\nu_-}h_n(\bigwedge_{\ell | n_+}\varphi_\ell)((1-\tau)u_0 \wedge \cdots \wedge u_{\nu_+}).$$
\end{lemma}

\begin{proof}[Proof of Theorem \ref{mrthm}]
By Theorem \ref{hnr2}, we have an equality 
$$\N_{L'/L}^{(1,\nu_+)}(\varep_T')=(-1)^{\nu_+}(\bigwedge_{\ell | n_+}\varphi_\ell)(\varep_T)$$
in $L^\times \otimes_\bbZ Q(L'/L)^{\nu_+}\otimes_\bbZ \bbZ[\frac12]$. 
From this and Lemma \ref{Tlem}, we deduce that an equality
$$(-1)^{\nu_+}\sum_{\sigma \in G(L'/L)}\sigma \N_{L(\mu_n)/L'}(\alpha_n) \otimes \sigma^{-1}=(-1)^{\nu_++1}2^{\nu_-}h_n(\bigwedge_{\ell | n_+}\varphi_\ell)((1-\tau)u_0 \wedge \cdots \wedge u_{\nu_+})$$
holds in $L^\times \otimes_\bbZ Q(L'/L)^{\nu_+}\otimes_\bbZ \bbZ[\frac12]$. By Lemma \ref{corrlem}, we have 
$$\theta_n=-2^{\nu_-}h_nR_n  \quad \mbox{in}\quad F^\times\otimes_\bbZ I_n^{\nu_+}/I_n^{\nu_+ +1}\otimes_\bbZ \bbZ[\frac12].$$
\end{proof}

\subsection{Proofs of Proposition \ref{propdes} and Lemma \ref{Tlem}} \label{proof2}
In this subsection, we give the proofs of Proposition \ref{propdes} and Lemma \ref{Tlem}.

\begin{proof}[Proof of Proposition {\rm{\ref{propdes}}}]
(Compare \cite[Lemma 8.1 and Proposition 9.4]{D}.)
It is known that 
$$\varep_T'=\N_{\bbQ(\mu_{nf})^+/L'}(\delta_T(1-\zeta_{nf})),$$
where $\delta_T=\prod_{\ell \in T}(1-\ell\Fr_\ell^{-1})$ (see \cite[\S 4.2]{P}).
Put 
$$G_n=\Gal(L(\mu_n)/L),$$
and 
$$\xi_n=\delta_T\sum_{\sigma \in G_n}\sigma \N_{\bbQ(\mu_{nf})/L(\mu_n)}(1-\zeta_{nf})\otimes \sigma^{-1} \in \cO_{L(\mu_n),\Sigma,T}^\times \otimes_\bbZ \bbZ[G_n].$$
It is easy to see that 
$$\pi(\xi_n)=2\sum_{\sigma \in G(L'/L)}\sigma \varep_T'\otimes \sigma^{-1},$$
where $\pi : \bbZ[G_n]\rightarrow \bbZ[G(L'/L)]$ is the natural projection. Hence, it is sufficient to prove that 
$$\xi_n \in \cO_{L(\mu_n),\Sigma,T}^\times \otimes_\bbZ I_n^{\nu_+},$$
and 
$$\xi_n \in \cO_{L,\Sigma,T}^\times \otimes_\bbZ I_n^{\nu_+}/I_n^{\nu_++1}\otimes_\bbZ \bbZ[\frac12].$$
We prove this by induction on $\nu_+$. When $\nu_+=0$, there is nothing to prove. When $\nu_+ >0$, decompose
$$G_n \simeq G_{n_-} \times G_{n_+},$$
where $G_{n_\pm}=\prod_{\ell | n_\pm}G_\ell$ and $G_\ell=\Gal(L(\mu_\ell)/L)$. Each $\sigma \in G_n$ is uniquely written as  
$$\sigma=\sigma_- \prod_{\ell|n_+}\sigma_\ell,$$
where $\sigma_-\in G_{n_-}$, and $\sigma_\ell \in G_\ell$. We compute 
$$\delta_T\sum_{\sigma \in G_n}\sigma \N_{\bbQ(\mu_{nf})/L(\mu_n)}(1-\zeta_{nf})\otimes \sigma_-^{-1}\prod_{\ell |n_+}(\sigma_\ell^{-1}-1) 
=\xi_n +\sum_{d|n_+, d\neq n_+}(-1)^{\nu(n_+/d)}\xi_{n_-d}\prod_{\ell |n_+/d}(1-\Fr_\ell^{-1}), $$
where $\nu(n_+/d)$ is the number of prime divisors of $n_+/d$.
From this and the inductive hypothesis, we have $\xi_n \in \cO_{L(\mu_n),\Sigma,T}^\times \otimes_\bbZ I_n^{\nu_+}$. Fix a generator 
$\gamma_\ell$ of $G_\ell$. In $\cO_{L(\mu_n),\Sigma,T}^\times \otimes_\bbZ I_n^{\nu_+}/I_n^{\nu_++1}$, we have 
$$\delta_T\sum_{\sigma \in G_n}\sigma \N_{\bbQ(\mu_{nf})/L(\mu_n)}(1-\zeta_{nf})\otimes \sigma_-^{-1}\prod_{\ell |n_+}(\sigma_\ell^{-1}-1)
= (-1)^{\nu_+} D_{n_+} \delta_T \N_{\bbQ(\mu_{nf})/L(\mu_{n_+})}(1-\zeta_{nf})\otimes \prod_{\ell |n_+}(\gamma_\ell-1),$$
where $D_{n_+} \in \bbZ[G_{n_+}]$ is the Kolyvagin's derivative operator, defined by 
$$D_{n_+}=\prod_{\ell|n_+}(\sum_{i=1}^{\ell-2}i\gamma_\ell^i).$$
Since we have the decomposition 
$$I_n^{\nu_+}/I_n^{\nu_++1} \simeq <\prod_{\ell|n_+} (\gamma_\ell-1)>_\bbZ \oplus \cI_n^{\rm{old}},$$
where $\cI_n^{\rm{old}}$ is a subgroup of $I_n^{\nu_+}/I_n^{\nu_++1}$, 
and the isomorphism 
$$<\prod_{\ell|n_+} (\gamma_\ell-1)>_\bbZ  \stackrel{\sim}{\longrightarrow} \bigotimes_{\ell|n_+}G_\ell \quad ; \quad 
\prod_{\ell|n_+}(\gamma_\ell -1) \mapsto \bigotimes_{\ell|n_+}\gamma_\ell,$$
(see \cite[Proposition 4.2 (i) and (iv)]{MR2}), it is sufficient to show that 
$$D_{n_+} \delta_T \N_{\bbQ(\mu_{nf})/L(\mu_{n_+})}(1-\zeta_{nf}) \in \cO_{L,\Sigma,T}^\times/(\cO_{L,\Sigma,T}^\times)^m,$$
where $m$ is the greatest odd common divisor of $\{\ell-1 \ | \ \ell|n_+ \}$. Note that 
$$\bigotimes_{\ell | n_+}G_\ell \otimes_\bbZ \bbZ[\frac12] \simeq \bbZ/m\bbZ.$$
It is well-known that 
$$D_{n_+} \delta_T \N_{\bbQ(\mu_{nf})/L(\mu_{n_+})}(1-\zeta_{nf}) \in (\cO_{L(\mu_{n_+}),\Sigma,T}^\times/(\cO_{L(\mu_{n_+}),\Sigma,T}^\times)^m)^{G_{n_+}},$$
(see \cite[Lemma 2.1]{R1}, \cite[Lemma 6.2]{D} or \cite[Lemma 4.4.2 (i)]{R}), hence the claim follows if we show that $H^1(G_{n_+},\cO_{L(\mu_{n_+}),\Sigma,T}^\times)=0$ for sufficiently large $\Sigma$. If $\Sigma$ is large enough, then we have the exact sequence 
$$0\longrightarrow \mathcal{O}_{L(\mu_{n_+}),\Sigma,T}^\times \longrightarrow \cO_{L(\mu_{n_+}),\Sigma}^\times \longrightarrow \bigoplus_{w\in T_{L(\mu_{n_+})}}\bbF_w^\times\longrightarrow 0,$$
where $\bbF_w^\times$ denotes the residue field at $w$. 
Since $ \bigoplus_{w\in T_{L(\mu_{n_+})}}\bbF_w^\times$ is a cohomologically-trivial $G_{n_+}$-module, the above exact sequence shows that $H^1(G_{n_+},\cO_{L(\mu_{n_+}),\Sigma,T}^\times)=H^1(G_{n_+},\cO_{L(\mu_{n_+}),\Sigma}^\times)$. Since $\Sigma$ is large enough, we have the exact sequence
$$0\longrightarrow \cO_{L(\mu_{n_+}),\Sigma}^\times \longrightarrow L(\mu_{n_+})^\times \stackrel{\bigoplus_w \ord_w}{\longrightarrow} \bigoplus_{w \notin \Sigma_{L(\mu_{n_+})}}\bbZ \longrightarrow 0.$$
From this, we see that $H^1(G_{n_+},\cO_{L(\mu_{n_+}),\Sigma}^\times)=0$. Hence we have $H^1(G_{n_+},\cO_{L(\mu_{n_+}),\Sigma,T}^\times)=0$. 
\end{proof}

\begin{remark} \label{remkoly}
Consider the following composite map: 
$$L^\times \otimes_\bbZ Q(L'/L)^{\nu_+} \otimes_\bbZ \bbZ[\frac12] \stackrel{\sim}{\longrightarrow} L^\times \otimes_\bbZ I_n^{\nu_+}/I_n^{\nu_++1}\otimes_\bbZ \bbZ[\frac12]$$
$$  \longrightarrow L^\times \otimes_\bbZ <\prod_{\ell | n_+}(\gamma_\ell -1)>_\bbZ \otimes_\bbZ \bbZ[\frac12] \stackrel{\sim}{\longrightarrow} L^\times /(L^\times)^m,$$
where the first isomorphism is $\pi^{-1}$, the second arrow is the projection, and the last isomorphism is induced by 
$$ <\prod_{\ell | n_+}(\gamma_\ell -1)>_\bbZ {\longrightarrow} \bbZ/m \bbZ \quad ; \quad \prod_{\ell | n_+}(\gamma_\ell -1) \mapsto 1.$$
If $n=n_+$ and put $\nu=\nu_+$, then the above proof shows that the image of $2\N_{L'/L}^{(1,\nu)}(\varep_T')$ under this map coincides with $(-1)^{\nu}D_{n}\varep_T'$. Hence, one can regard that the ``higher norm operator'' $\N_{L'/L}^{(1,\nu)}$ is a generalization of Kolyvagin's derivative operator $D_{n}$. This observation is originally due to Darmon (\cite[Proposition 9.4]{D}).
\end{remark}


\begin{proof}[Proof of Lemma {\rm{\ref{Tlem}}}]
(i) From
$$2 \varep_T'= \delta_T \N_{\bbQ(\mu_{nf})/L'}(1-\zeta_{nf}),$$
we obtain
$$2\sum_{T\in\cT}a_T \varep_T'=2\N_{\bbQ(\mu_{nf})/L'}(1-\zeta_{nf})$$
(see Proposition \ref{propT}). We compute
\begin{eqnarray}
(1-\tau)\N_{\bbQ(\mu_{nf})/L'}(1-\zeta_{nf}) &=& \N_{L(\mu_n)/L'}((1-\tau)\N_{\bbQ ( \mu_{nf})/L(\mu_{n})}(1-\zeta_{nf})) \nonumber \\
&=&\N_{L(\mu_n)/L'}(\alpha_n), \nonumber
\end{eqnarray}
hence we have 
$$(1-\tau)\sum_{T\in\cT}a_T\varep_T'=\N_{L(\mu_n)/L'}(\alpha_n) \quad \mbox{in} \quad L'^\times/\{ \pm 1 \}.$$
(ii) By Lemma \ref{reginj}, $R_V$ is injective on $e_\chi(\bbQ \otimes_\bbZ \bigwedge_{\cG_L}^{\nu_++1}\cO_{L,S}^\times)$, so it is sufficient to prove 
that 
$$R_V((1-\tau)\sum_{T\in\cT}a_T\varep_T)=(-1)^{\nu_++1}2^{\nu_-}h_nR_V((1-\tau)u_0 \wedge \cdots \wedge u_{\nu_+}).$$
By the characterization of $\varep_T$, the left hand side is equal to $2(1-\tau)\Theta_{L,S}^{(\nu_++1)}(0)$. Using the 
well-known class number formulas for $n$-truncated Dedekind zeta functions of $L$ and $\bbQ$ (see \cite[\S 1]{G}), we have 
$$2(1-\tau)\Theta_{L,S}^{(\nu_++1)}(0)=4h_ne_\chi \frac{R_{L,n}}{R_{\bbQ,n}},$$
where $R_{L,n}$ and $R_{\bbQ,n}$ are the usual $n$-regulators for $L$ and $\bbQ$ respectively. In Lemma \ref{compute}, we will prove an equality 
$$e_\chi R_{L,n}=(-1)^{\nu_++1}2^{\nu_- -1}R_{\bbQ,n}e_\chi R_V(u_0\wedge \cdots \wedge u_{\nu_+}).$$
Hence we have 
$$2(1-\tau)\Theta_{L,S}^{(\nu_++1)}(0)=(-1)^{\nu_++1}2^{\nu_-}h_nR_V((1-\tau)u_0\wedge\cdots\wedge u_{\nu_+}),$$
which completes the proof. 
\end{proof}

\begin{lemma} \label{compute}
$$e_\chi R_{L,n}=(-1)^{\nu_++1}2^{\nu_- -1}R_{\bbQ,n}e_\chi R_V(u_0\wedge \cdots \wedge u_{\nu_+}).$$
\end{lemma}

\begin{proof}(Compare the proof of \cite[Theorem 3.5]{R2}.)
There is an exact sequence of abelian groups:
$$0 \longrightarrow \bbZ[\frac1n]^\times/\{ \pm 1\} \longrightarrow \cO_L[\frac1n]^\times/\{\pm 1\} \stackrel{1-\tau}{\longrightarrow} (1-\tau)\cO_L[\frac1n]^\times \longrightarrow 0.$$
Since $(1-\tau)\cO_L[\frac1n]^\times$ is torsion-free (see \cite[Lemma 3.2 (ii)]{MR2}), this exact sequence splits. So we can choose $\eta_1,\ldots,\eta_\nu \in \bbZ[\frac1n]^\times$ so that $\{ \eta_1,\ldots,\eta_\nu,u_0,\ldots,u_{\nu_+}\}$ is a basis of $\cO_L[\frac1n]^\times/\{\pm1\}$ ($\nu$ is the number of prime divisors of $n$). Write $n_-=\prod_{i=1}^{\nu_-}\ell_i'$, where $\ell_i'$ is a prime number. Let $\lambda_i'$ be the (unique) place of $L$ lying above $\ell_i'$. We compute the regulator $R_{L,n}$ with respect to the basis $\{  \eta_1,\ldots,\eta_\nu,u_0,\ldots,u_{\nu_+} \}$ of $\cO_L[\frac1n]^\times/\{\pm1\}$ and the places $\{ \lambda_2',\ldots,\lambda_{\nu_-}',\lambda_0^\tau,\ldots,\lambda_{\nu_+}^\tau,\lambda_0,\ldots,\lambda_{\nu_+} \}$:
$$R_{L,n}=\pm \det \left(
\begin{array}{ccc}
	\log|\eta|_{\lambda'} &\log|\eta|_{\lambda^\tau} & \log|\eta|_\lambda \\ 
	\log|u|_{\lambda'} & \log|u|_{\lambda^\tau} &\log|u|_\lambda
\end{array}
\right),$$
where we omit the subscript, for simplicity (for example, $\log|\eta|_{\lambda'}$ means the $\nu \times (\nu_{-}-1)$-matrix $(\log|\eta_i|_{\lambda_j'})_{{1 \leq i \leq \nu},{2\leq j \leq \nu_-}}$). We may assume that the sign of the right hand side is positive (replace $\eta_1$ by $\eta_1^{-1}$ if necessary). We compute 
\begin{eqnarray}
\det \left(
\begin{array}{ccc}
	\log|\eta|_{\lambda'} &\log|\eta|_{\lambda^\tau} & \log|\eta|_\lambda \\ 
	\log|u|_{\lambda'} & \log|u|_{\lambda^\tau} &\log|u|_\lambda
\end{array}
\right) &=& \det \left(
\begin{array}{ccc}
	\log|\eta|_{\lambda'} &\log|\eta|_{\lambda} & \log|\eta|_\lambda \\ 
	\log|u|_{\lambda'} & \log|u|_{\lambda^\tau} &\log|u|_\lambda
\end{array}
\right) \nonumber \\
&=& \det \left(
\begin{array}{ccc}
	\log|\eta|_{\lambda'} &\log|\eta|_{\lambda} & 0 \\ 
	\log|u|_{\lambda'} & \log|u|_{\lambda^\tau} &\log|u|_\lambda -\log|u|_{\lambda^\tau} 
\end{array}
\right) \nonumber \\
&=& \det (
\begin{array}{cc}
	\log|\eta|_{\lambda'} & \log|\eta|_\lambda
\end{array}
)
\det (\log|u|_\lambda-\log|u|_{\lambda^\tau})\nonumber \\
&=&\det (
\begin{array}{cc}
	2\log|\eta|_{\ell'} & \log|\eta|_\ell
\end{array}
)
\det(\log|(1-\tau)u|_\lambda) \nonumber \\
&=& 2^{\nu_--1}R_{\bbQ,n}\det(\log|(1-\tau)u|_\lambda). \nonumber
\end{eqnarray}
Hence we have 
\begin{eqnarray}
e_\chi R_{L,n}=2^{\nu_--1}R_{\bbQ,n} e_\chi \det (\log|(1-\tau)u|_\lambda). \label{eqreg}
\end{eqnarray}
On the other hand, we compute 
\begin{eqnarray}
e_\chi R_V(u_0 \wedge \cdots \wedge u_{\nu_+})&=&(-1)^{\nu_++1}e_\chi \det(\log|u|_\lambda + \log|\tau(u)|_\lambda \tau) \nonumber \\
&=& (-1)^{\nu_++1}e_\chi \det(\log|(1-\tau)u|_\lambda +(1+\tau)\log|\tau(u)|_\lambda) \nonumber \\ 
&=& (-1)^{\nu_++1}e_\chi \det(\log|(1-\tau)u|_\lambda), \nonumber
\end{eqnarray}
where the first equality follows by noting that $R_V=\bigwedge_{0 \leq i \leq \nu_+}(-\log|\cdot|_{\lambda_i}-\log|\tau(\cdot)|_{\lambda_i}\tau)$ by definition (see \S \ref{not}), and 
the last equality follows from $e_\chi (1+\tau)=0$. Hence, by (\ref{eqreg}), we have the desired equality
$$e_\chi R_{L,n}=(-1)^{\nu_++1}2^{\nu_--1}R_{\bbQ,n}e_\chi R_V(u_0\wedge \cdots \wedge u_{\nu_+}).$$
\end{proof}

\section*{Acknowledgement}
The author would like to thank Prof. Masato Kurihara for helpful comments and advice. He also wishes to thank Prof. David Burns for reading the manuscript of this paper and giving him helpful remarks.

\end{document}